\documentclass[12pt]{amsart} 
\usepackage[left=3cm,top=2.5cm,bottom=2.5cm,right=3cm]{geometry}
\usepackage[oldstyle]{libertine}%
\usepackage{amssymb, amsmath, amsthm}

\usepackage{graphicx,xspace}
\usepackage{epsfig}
\usepackage{enumitem}
\usepackage[usenames,dvipsnames]{xcolor}
\usepackage{tikz}
\usepackage{gensymb}
\usepackage{braket}
\usepackage[T1]{fontenc}
\usepackage[utf8]{inputenc}
\makeatletter
\renewcommand*\libertine@figurestyle{LF}
\makeatother
\usepackage[libertine,libaltvw,liby]{newtxmath}
\makeatletter
\renewcommand*\libertine@figurestyle{OsF}
\usepackage{subcaption}
\usepackage{mathtools,thmtools}

\usepackage{mathrsfs}
\usepackage{yfonts}
\definecolor{green}{RGB}{0,127,0}
\definecolor{red}{RGB}{191,0,0}
\usepackage[colorlinks,citecolor=red,linkcolor=green,pagebackref=true,hypertexnames=false]{hyperref}
\usepackage{todonotes}

\usepackage[capitalize]{cleveref}

\theoremstyle{plain}
\newtheorem{theorem}{Theorem}[section]
\newtheorem{lemma}[theorem]{Lemma}
\newtheorem{corollary}[theorem]{Corollary}
\newtheorem{proposition}[theorem]{Proposition}

\newtheorem{definition-lemma}[theorem]{Definition-Lemma}

\theoremstyle{definition}
\newtheorem{definition}[theorem]{Definition}
\newtheorem{remark}[theorem]{Remark}

\theoremstyle{plain}
\newtheorem{introthm}{Theorem}

\newcommand{\hh}{\mathbf{h}}
\newcommand{\pp}{\mathbf{p}}
\newcommand{\PP}{\mathbf{P}}
\newcommand{\qq}{\mathbf{q}}
\newcommand{\QQ}{\mathbf{Q}}
\newcommand{\xx}{\mathbf{x}}
\newcommand{\uu}{\mathbf{u}}
\newcommand{\g}{\mathbf{g}}
\newcommand{\Z}{\mathbb{Z}}

\newcommand{\Q}{\mathbb{Q}}
\newcommand{\R}{\mathbb{R}}
\newcommand{\CP}{{\mathbb{C}\mathbb{P}}}

\newcommand{\J}{\mathsf{J}}

\newcommand{\HurGbg}{H^{(b)}_{G;g}}
\newcommand{\HurGbbg}{H^{(\bb)}_{G;g}}
\newcommand{\HurGg}{H^{(0)}_{G;g}}
\newcommand{\la}{\lambda}
\newcommand{\dz}[1]{\mathsf{d}z_{#1}}
\newcommand{\dxz}[1]{\mathsf{d}x(z_{#1})}
\newcommand{\al}{{(\alpha)}}
\newcommand{\bb}{\mathfrak{b}}
\newcommand{\FI}{\mathcal{F}_<}

\newcommand{\splus}{\!+\!}
\newcommand{\sminus}{\!-\!}

\DeclareMathOperator{\Aut}{Aut}
\DeclareMathOperator{\Id}{Id}
\DeclareMathOperator{\hook}{hook}
\DeclareMathOperator{\wt}{wt}
\DeclareMathOperator{\supp}{supp}
\DeclareMathOperator{\Ad}{Ad}

\AtBeginDocument{%
   \def\MR#1{}
}

\title[$b$-Hurwitz numbers from Whittaker vectors for $\mathcal{W}$-algebras]{$b$-Hurwitz numbers from Whittaker vectors for $\mathcal{W}$-algebras}

	\author[N. K. Chidambaram]{Nitin K. Chidambaram}
\address{
	School of Mathematics, University of Edinburgh, James Clerk Maxwell Building, Peter Guthrie Tait Rd, Edinburgh EH9 3FD, U.K. %
}
\address{
  \textit{  Present address:} Departamento de Matem\'aticas Fundamentales, UNED,
    Calle de Juan del Rosal, 28040 Madrid, Spain. 
}
\email{nitin.chidambaram@mat.uned.es}

\author[M.~Dołęga]{Maciej Dołęga}
\address{
Institute of Mathematics, 
Polish Academy of Sciences, 
ul. Śniadeckich 8, 
00-956 Warszawa, Poland.
}
\email{mdolega@impan.pl}

\author[K.~Osuga]{Kento Osuga}
\address{Graduate School of Mathematical Sciences, University of Tokyo, 3-8-1 Komaba, Meguro, Tokyo, 153-0041, Japan}
\address{
\textit{  Present address:} Graduate School of Mathematics \& Kobayashi--Maskawa Institute for the Origin of Particles and the Universe, Nagoya University, Furocho, Nagoya, Aichi, 464-8601, Japan}
\email{osuga@math.nagoya-u.ac.jp}

\begin{document}
\emergencystretch 3em %%%%%%%%%%%% it makes text not go into the right margin
\begin{abstract}
We show that rationally weighted $b$-Hurwitz numbers  are obtained by taking an explicit limit of a Whittaker vector for the $\mathcal{W}$-algebra of type $A$. As a consequence, we show that the generating function satisfies an infinite set of finite degree differential operators that determines it uniquely.  We provide  an interpretation of the associated Whittaker vector in terms of generalized branched coverings that might be of independent interest. Our result is new even in the special case $b=0$ that corresponds to classical hypergeometric Hurwitz numbers, and implies that they are governed by the Eynard--Orantin topological recursion. This gives an independent proof of the recent result of Bychkov--Dunin-Barkowski--Kazarian--Shadrin.
\end{abstract}

\maketitle

\tableofcontents

\vspace{-.9 cm}
\section{Introduction}\label{sec:Introduction}

Hurwitz theory, in its simplest form, counts
branched coverings of algebraic curves. First
considered in 1891 by Hurwitz~\cite{Hurwitz1891} who showed that Hurwitz numbers
can be expressed in terms of counting permutation factorizations,  Hurwitz
theory has recently been found to have  fascinating connections to
integrable hierarchies, Gromov–-Witten theory, and topological
recursion \cite{EkedahlLandoShapiroVainshtein2001,OkounkovPandharipande2006,GouldenJackson2008,BouchardMarino2008}. In the last two decades, various
generalizations of simple Hurwitz numbers, including Grothendieck dessins
d'enfants/constellations~\cite{BousquetSchaeffer2000}, monotone Hurwitz numbers
\cite{GouldenGuayPaquetNovak2014} and orbifold Hurwitz
numbers~\cite{JohnsonPandharipandeTseng2011},  have been 
studied in many different contexts. 

An approach unifying all these different classes was introduced in \cite{Guay-PaquetHarnad2017} which defined $G$-weighted Hurwitz
  numbers $\HurGg(\mu_1,\dots,\mu_n)$, where the weight $G(z):= \sum_{i \geq 0}g_i z^i$
is a formal power series. The generating function of weighted Hurwitz numbers is a
tau-function of hypergeometric type of
the KP (or more generally the $2$-Toda) integrable hierarchies~\cite{OrlovScherbin2000}, and  is
given by the following explicit formula: 
\begin{equation}
  \label{eq:TauBGWeight}
 \tau_G^{(0)} := \sum_{\substack{n \geq 0,\\\lambda \vdash n}} 
    \check{s}_\lambda(\pp)\prod_{\square
                  \in \lambda}\frac{G(\hbar\cdot c(\square))}{\hbar}
                = \exp\left(\sum_{\substack{g
          \in
          \Z_{\geq 0},\\ n \in \Z_{\geq
            1}}}\hbar^{2g-2+n}\sum_{\mu \colon \ell(\mu)
          = n}\HurGg(\mu_1,\dots,\mu_n)
          \cdot p_\mu\right),
      \end{equation}
      where $\check{s}_\lambda(\pp)$ is an appropriately normalized
      Schur function expressed in the basis of power-sum symmetric functions $\pp = (p_k)_{k \geq
      1}$, and $c(\square) := x-y$ is
      the content of the box $\square = (x,y)$.

      \subsection{$b$-deformed Hurwitz numbers}
      \label{sub:b-Hur}

      In this paper, we are interested in a refinement of  Hurwitz theory, known as \emph{$b$-Hurwitz numbers} defined by Chapuy and the second-named author \cite{ChapuyDolega2022}. $G$-weighted $b$-Hurwitz numbers $\HurGbg (\mu_1,\dots,\mu_\ell)$ are  a  one-parameter deformation of classical weighted Hurwitz numbers, whose generating function is given by
\begin{equation*}
 \tau_G^{(b)} := \sum_{\substack{n \geq 0,\\\lambda \vdash n}} 
    \check{J}_\lambda^{(1+b)}(\pp)\prod_{\square
                  \in \lambda}\frac{G(\hbar\cdot c_{1+b}(\square))}{\hbar}
                = \exp\left(\sum_{\substack{g
          \in
         \frac12 \Z_{\geq 0},\\ n \in \Z_{\geq
            1}}}\frac{\hbar^{2g-2+n}}{1+b}\sum_{\mu \colon \ell(\mu)
          = n}\HurGbg(\mu_1,\dots,\mu_n)
          \cdot p_\mu\right),
      \end{equation*} where $c_{1+b}(x,y) := (b+1)\cdot x-y -b$ is the $b$-deformation of the content, and $\check{J}^{(1+b)}_\lambda(\pp)$ is an appropriately normalized Jack polynomial with the condition  $\check{J}_\lambda^{(1)}(\pp)= \check{s}_\lambda(\pp)$. In particular, $H^{(b)}_{G;g}$ agree with the $H^{(0)}_{G;g}$ of \eqref{eq:TauBGWeight} when $b=0$. The appearance of Jack functions deforming Schur functions away from $b=0$ is motivated by the analogous relation in matrix models that appears when one studies $\beta$-deformations (where $\beta  = \frac{1}{1+b}$).

      \cite{ChapuyDolega2022} provide an enumerative/combinatorial definition of  $b$-Hurwitz numbers that interpolates between classical (complex) Hurwitz theory (at $b=0$) and its non-orientable (real) version (at $b =1$). As a consequence, they prove that for any $\mu_1,\dots,\mu_n$, the quantity
      $|\mu|\cdot\HurGbg(\mu_1,\dots,\mu_n) \in \Z_{\geq 0}[b][g_0,g_1,\dots]$ is a polynomial in $b$ with non-negative integer coefficients. This statement was conjectured in several special cases previously. Weighted $b$-Hurwitz numbers  encompass many models appearing in mathematical physics, enumerative geometry, and
      combinatorics, including Jacobi, Laguerre and Gaussian $\beta$-ensembles \cite{BonzomChapuyDolega2023,Ruzza2023}, $\beta$-deformations of the Harish-Chandra/Itzykson--Zuber (HCIZ)
      integral and the Br\'ezin--Gross--Witten integral \cite{BonzomChapuyDolega2023}.

      \subsection{(Decoupled) cut-and-join equations}
 
     One of the main tools in understanding the structure of classical
     Hurwitz numbers are cut-and-join equations that determine $\HurGg$
     uniquely. In examples where
     the weight $G(z)$ is relatively simple, one can understand these
     equations in terms of the Jucys--Murphy elements that further lead
     to methods employing the representation theory of the
     symmetric group and so-called semi-infinite wedge representations (see for
     instance~\cite{BorotDoKarevLewanskiMoskovsky2023} and references therein).
     These techniques can then be used to prove various structural properties of Hurwitz numbers --- for instance,  the Eynard--Orantin topological
     recursion \cite{EynardMulaseSafnuk2011,DoDyerMathews2017}, polynomiality and 
     ELSV-type formulas \cite{GouldenJacksonVakil2005,BorotDoKarevLewanskiMoskovsky2023}.

 On the other hand, finding cut-and-join equations for arbitrary $b$ is fundamentally more challenging than for $b=0$, as the aforementioned tools, such as permutation factorizations, the boson-fermion correspondence, and the semi-infinite wedge representations, do not extend to generic $b$. The approach initiated in~\cite{ChapuyDolega2022} is to construct  a family of operators (directly on the bosonic side) whose action on Jack polynomials is well controlled, and to relate their combinatorial properties to the properties of generalized branched coverings. We pursue this approach further by combining the operators constructed in~\cite{ChapuyDolega2022} with another remarkable family of operators constructed by Nazarov and Sklyanin in~\cite{NazarovSklyanin2013}, and by systematically analyzing the interplay of their algebraic structures. This enables us to derive a family of cut-and-join equations for weighted $b$-Hurwitz numbers of finite degree for rational weights $G$. In fact, we go further and decouple the cut-and-join equations into an infinite set of finite order differential constraints that uniquely determine the generating function $\tau^{(b)}_G$. 

     We obtain the following set of differential constraints on $\tau^{(b)}_G$, which can be interpreted as coming from a set of Virasoro/$\mathcal W$-constraints (see \cref{sec:introW} below).
    \begin{introthm} \label{thm:introcutandjoin} The generating function  $\tau^{(b)}_G$ of
        $b$-Hurwitz numbers weighted by $G(z) =
        \frac{\prod_{i=1}^p\left(P_i+
            z\right)}{\prod_{i=1}^q\left(Q_i-z\right)}  $, for any
        integers $p,q \geq 0$, is uniquely determined by the following differential operator constraints\footnote{In the body of the paper, we work with a  reparametrization of the generating function $\tau^{(b)}_G$, which is denoted by $\tau^{(\bb)}_G$. The shift $\hbar \to \hbar \sqrt{1+b}$ appearing in equation \eqref{eq:tauconstraints} is a consequence of this reparametrization.}:
    \begin{equation}\label{eq:tauconstraints}
       \Biggl( \sum_{\gamma \in \Gamma^{\geq 0}_{(0,k)\to (p,0)}} \widetilde{\wt}\left(\gamma| \left(\mathbf P \right)\right)  
      + (-1)^{q+1} \sum_{\gamma \in \Gamma^{\geq 0}_{(0,k)\to (q+1,-1)}}\widetilde{\wt}\left(\gamma|\left(-\mathbf Q,0\right) \right) \Biggr) \Biggl|_{\hbar \to \hbar \sqrt{1+b}}  \tau^{(b)}_{G} = 0, \,\,\,\, k \geq 0.
    \end{equation} The sums are over lattice paths $\gamma$ of a specific shape, and $\widetilde{\wt}(\gamma | \bullet)$ is an explicit differential operator in the power sums $(p_k)_{k \geq 1}$ associated to the path $\gamma$ (see \eqref{eq:tildewt}). In particular, \eqref{eq:tauconstraints} is an explicit differential equation of degree $\textup{max}(p,q+1)$  in the power sum variables. 
    \end{introthm} 
    For example, the rescaled operators \eqref{eq:tauconstraints} associated to the weight $G(z) = \frac{P+z}{Q-z}$ are quadratic and using the notation, for $k >0$,
    \begin{equation*}
      \J_{k} =  \hbar k (1+b) \cdot  \partial_{p_{k}}, \qquad \qquad \J_{-k} =   \hbar p_{k}, \qquad  \qquad \J_0 = 0,
    \end{equation*}
    are given explicitly as
    \[
          \left(\J_k + \delta_{k,0} P + \sum_{\ell \geq 0 } \J_{k-\ell} \J_{\ell+1} + (\hbar b k - Q)\J_{k+1}\right) \tau^{(b)}_{G} = 0, \qquad k \geq 0.
    \] 
    It is easy to check that at $\hbar = 1$,  the above equations are a set of Virasoro constraints (up to rescaling by $\sqrt{1+b}$).  When the weight $G(z) = \frac{(P_1+z)(P_2+z)}{(Q_1-z)(Q_2-z)}$, we have the following cubic constraints instead:
    \begin{multline*}
          \left(\sum_{\ell \geq 0 } \J_{k-\ell} \J_{\ell} + (P_1+P_2+\hbar b k)\J_{k}+ P_1P_2 \delta_{k,0}- \sum_{\ell,m \geq 0 } \J_{k-\ell} \J_{\ell-m} \J_{m+1} 
          \right.
          \\
          \left.
          + \sum_{\ell \geq 0} (Q_1+Q_2- \hbar b (k+ \ell))\J_{k-\ell}\J_{\ell+1} - (Q_1-\hbar b k)(Q_2-\hbar b k) \J_{k+1}\right) \tau^{(b)}_{G} = 0, \quad k \geq 0.
    \end{multline*}

     We note that the cut-and-join equations given in \cite{ChapuyDolega2022} are of infinite order for rational weights (decoupling those cut-and-join-type equations are
 highly non-trivial even in the  case of polynomial weight $G$, see
\cite{BonzomNador2023} for the case of cubic polynomial
 weight), while ours are of finite order and extend the  one derived in \cite{BonzomChapuyDolega2023} for monotone Hurwitz numbers to any rational weight. Multiplying \eqref{eq:tauconstraints} by $p_{k+1}$ on the left and summing over $k \geq 0$ gives our  cut-and-join equation (proved in \cref{thm:rationalweight}) in a familiar form.

\subsection{\texorpdfstring{$\mathcal W$}{W}-algebra representation} \label{sec:introW}
     The decoupling of the cut-and-join equation into an infinite set of differential constraints is reminiscent of how the cut-and-join equation is built out of Virasoro constraints in the monotone case \cite{BonzomChapuyDolega2023}. In fact, we prove \cref{thm:introcutandjoin} by unveiling the underlying structure of our cut-and-join equations as being closely related to the structure of a $\mathcal W$-algebra (a vertex algebra that generalizes the Virasoro algebra). An important insight is to use lattice paths as a combinatorial tool that simultaneously encodes  operators acting
on $\tau^{(b)}_G$ as well as the $\mathcal
W$-algebra structure. As a consequence we show that $\tau^{(b)}_G$ for any  rational weight $G(z)$ can be obtained from a  Whittaker vector for a $\mathcal W$-algebra of type $A$ via some simple substitutions and limits. 
     
      More concretely, we work with the principal $\mathcal W$-algebra  of  $\mathfrak{gl}_r $ at shifted level $\mathsf k + r-1 = -\frac{b}{\sqrt{1+b}} $, denoted $\mathcal W^{\mathsf
        k}(\mathfrak{gl}_r)$, and a specific representation of it  defined via \eqref{eq:Jrep}. Using Airy structure techniques (introduced in \cite{KontsevichSoibelman2018, BorotBouchardChidambaramCreutzigNoshchenko2024,BorotBouchardChidambaramCreutzig2024}), we construct a Whittaker vector $\mathcal Z$ for  $\mathcal W^{\mathsf k}(\mathfrak{gl}_r)$ satisfying the following constraints
     \begin{equation}\label{eq:Wconstintro}
     \mathsf W^i_k  \mathcal Z = \Omega_i \delta_{k,0} \mathcal Z\, \qquad \forall i \in [r], k \in \mathbb Z_{\geq 0},
     \end{equation} where $\big(\mathsf{W}^i_k\big)_{k\in \mathbb Z, i \in [r]}$ are the modes of a set of strong generators for $\mathcal W^{\mathsf k}(\mathfrak{gl}_r)$, and $\Omega_{i}$ is a certain function of the parameters appearing in the rational weight $G(z)$. Note that the vector $\mathcal Z$ is not the highest weight vector of the representation we are working in, but rather a Whittaker vector (see \cref{rem:Whittaker} for more details). An abbreviated version of one of our main results \cref{thm:taufromZ} is as follows.
     \begin{introthm}\label{theo:Main}
      Consider the generating function $ \tau^{(b)}_{G}$ of weighted $b$-Hurwitz numbers for an arbitrary rational weight, say $G(z) = \frac{\prod_{i=1}^p (P_i+z)}{\prod_{i=1}^q (Q_i-z)} $ for some $p,q\geq 0$. The function $ \tau^{(b)}_{G}$ is an (explicit) limit of the  Whittaker vector $\mathcal Z$ for $\mathcal W^{\mathsf k}(\mathfrak{gl}_r) $ where $ r = \max(p,q+2)$.
     \end{introthm} 
     Under this identification, the differential operator \eqref{eq:tauconstraints} annihilating $\tau^{(b)}_G$  is constructed out of the modes $\mathsf{W}^i_k$ of $\mathcal W^{\mathsf k}(\mathfrak{gl}_r)$, and \cref{thm:introcutandjoin} is a consequence of this result.

     A slight variant of \cref{theo:Main} states that the Gaiotto vector from the
        celebrated AGT conjecture~\cite{AldayGaiottoTachikawa2010, SchiffmannVasserot2013,MaulikOkounkov2019, BorotBouchardChidambaramCreutzig2024} is closely related to the generating function of certain weighted $b$-Hurwitz numbers, which might shed new light on the relation between Toda conformal blocks and the Nekrasov instanton partition function (see \cref{rem:AGT} for more details).

    \subsection{Topological recursion for $b=0$} \label{sec:introTR}
    
    In various contexts \cite{BorotBouchardChidambaramCreutzigNoshchenko2024, BorotBouchardChidambaramCreutzig2024}, it has been proved using the formalism of Airy structures that $\mathcal W$-constraints are closely related to the  topological recursion (TR) formalism, originally invented by
    Eynard and Orantin~\cite{EynardOrantin2007}. TR
    is a universal procedure which produces certain enumerative
    invariants  $\omega_{g,n}$, called correlators, for a fixed topology of genus $g$ with
    $n$ boundaries. These correlators are built recursively from the initial
    data of a spectral curve $(\Sigma,x,y,\omega_{0,2})$, 
    consisting of a Riemann surface $\Sigma$, two non-constant
    meromorphic functions $x,y$ on $\Sigma$, and a choice of a canonical
    symmetric bi-differential $\omega_{0,2}$ on $\Sigma \times \Sigma$.
    
\cite{BorotChidambaramUmer2025} proves that the  Whittaker vector
    $\mathcal Z$ at the self-dual level $\mathsf k = 1-r $, which
    corresponds to  $b=0$, can be
    computed using topological recursion on a certain spectral
    curve. Combining this with \cref{theo:Main}, and taking
    appropriate limits allow us to deduce that the classical
    rationally weighted Hurwitz numbers can be computed by TR, which was first proved in \cite{BychkovDuninBarkowskiKazarianShadrin2024}. This is one of the most important structural results in Hurwitz theory, and we provide a new and independent proof of it in \cref{sec:TR} from the $\mathcal{W}$-algebra perspective:

     \begin{introthm}
    	\label{theo:TRRational}
    	Consider the rational weight $G(z) = \frac{\prod_{i=1}^p
    		(P_i+z)}{\prod_{i=1}^q(Q_i-z)}$ for some $p,q \geq 0$. Let $\omega_{g,n}$ denote the correlators obtained from TR on the spectral curve $\left(\CP^1,x,\omega_{0,1} = y dx,\omega_{0,2}\right)$, with 
    	\[
    	x(z) = \frac{z}{G(z)}, \quad y(z) = G(z), \quad 	\omega_{0,2}(z_1,z_2) = \frac{\dz{1} \dz{2}}{(z_1-z_2)^2}.  \] Then, expanding the correlators $\omega_{g,n}(z_1,\ldots,z_n)$ near $z_1,\dots,z_n = 0$ gives $G$-weighted Hurwitz numbers:
    	\begin{equation*}
    		\label{eq:Correlators}
    		\omega_{g,n}(z_1,\dots,z_n) = \sum_{\mu_1,\dots,\mu_n \in \Z_{\geq
    				1}}\HurGg(\mu_1,\dots,\mu_n)|\Aut(\mu)|\prod_{i=1}^n\mu_ix_i(z)^{\mu_i-1}\dxz{i} \,.
    	\end{equation*}
    \end{introthm}
    Let us comment briefly on the history of this theorem. Before the breakthrough of \cite{AlexandrovChapuyEynardHarnad2020} where Hurwitz numbers weighted by arbitrary polynomials $G$ was proved to be governed by TR, each model of interest was treated case by case over a decade: \cite{BouchardMarino2008,EynardMulaseSafnuk2011} (simple
    Hurwitz numbers), \cite{DoDyerMatthews2017} (monotone),
    \cite{KramerPopolitovShadrin2022} (monotone orbifold) to name just
    a few. Finally, \cite{BychkovDuninBarkowskiKazarianShadrin2024} gave a uniform proof of a more general result, that in particular, covers the case of arbitrary rational weights $G$ given in \cref{theo:TRRational}.

    The existing proofs of \cref{theo:TRRational} rely heavily on the
    KP integrability of the generating function $\tau^{(0)}_G$ that
    is specific to the $b=0$ case and
    does not have even a conjectural extension to arbitrary $b$. In
    our proof, we use the underlying $\mathcal W$-algebra structure
    controlling Hurwitz theory for arbitrary $b$ as we prove in
    \cref{theo:Main}, and  the close relationship
    between  $\mathcal W$-algebras and the Eynard--Orantin topological
    recursion. 

    \subsection{Outlook} 
    The $\mathcal W$-algebra structure we uncover in this paper appears to be a fundamental structure underlying Hurwitz theory.  For arbitrary $b$, one may view this structure as a replacement for the  Fock space formalism and KP integrability, both of which only currently exist in the $b=0$ setting. Indeed, we demonstrate that cut-and-join-equations and topological recursion, which are key structural results in  Hurwitz theory, are consequences of this $\mathcal W$-algebra structure. Theorems \ref{thm:introcutandjoin} and \ref{theo:Main} are therefore significant steps towards understanding structural properties of $b$-Hurwitz numbers.

    Our results  raise the  natural question of what the correct generalization of the  Eynard--Orantin topological recursion that governs weighted $b$-Hurwitz numbers is. While the $\mathcal W$-algebra structure holds for any $b$, the analytic side that mimics the original construction of Eynard--Orantin is still under investigation. There are two differing approaches to extend the Eynard--Orantin topological recursion to arbitrary $b$. One approach,  developed in \cite{ChekhovEynardMarchal2011,BelliardEynard2019,BorotBouchardChidambaramCreutzig2024} is to replace the initial data of a spectral curve by a non-commutative spectral curve (essentially, a $D$-module on the spectral curve) and goes by the name of \textit{non-commutative topological recursion}. The underlying $\mathcal W$-algebra structure of this non-commutative topological recursion was explored thoroughly in  \cite{BorotBouchardChidambaramCreutzig2024} for a specific class of examples, which, in  Hurwitz theory corresponds to the case where $1/G(z)$ is a polynomial (see \cref{rem:AGT}). While it seems difficult to extract analytic properties of $b$-Hurwitz numbers directly from non-commutative TR, this connection  is an interesting problem that deserves further study.
    
    Another promising approach is the so-called \textit{refined topological recursion}. A refinement was first attempted by Chekhov and Eynard \cite{ChekhovEynard2006a} in the case of
    $\beta$-deformed matrix models, and its mathematical formulation for degree-$2$ curves was recently carried out by Kidwai and the third author in \cite{KidwaiOsuga2023,Osuga2024}. The initial data is upgraded to a refined spectral curve which comes with a new differential $\omega_{\frac12,1}$.  In subsequent work \cite{ChidambaramDolegaOsuga2026} we use our main result \cref{theo:Main} to prove that  refined TR computes the tau function of some interesting classes of weighted $b$-Hurwitz numbers, including the monotone case and the case of Gaussian, Laguerre and Jacobi $\beta$ ensembles.

\subsection{Notation} Given an integer $n \geq 1$, we will denote the set
$\{1,\dots,n\}$ by $[n]$, and given integers $2 \leq m \leq  n$, we
define the set $[m..n] := [n] \setminus [m-1]$ ($[0] :=
\emptyset$ by convention). We denote by $\FI(A,B)$ a
set of strictly increasing functions $f\colon A \rightarrow B$, i.e.~ $f(a) < f(a')$ for all $a < a'$. We use a bold face to indicate
that we are working with a family of indeterminates indexed by positive integers,
e.g.~ $\pp = (p_1,p_2,\dots), \g = (g_1,g_2,\dots), \xx =
(x_1,x_2,\dots)$, etc. For a ring $R$ we denote by $R[\xx]$, $R(\xx)$,
$R\llbracket \xx \rrbracket$, and $R(\!(\xx)\!)$ the polynomial, rational, formal
power series, and formal Laurent series, respectively, rings over $R$. For a function $F$ of the variables $\mathbf p$, we use the notation $[p_{i_1}\cdots p_{i_n}]F$ to denote the coefficient of
the monomial $p_{i_1}\cdots p_{i_n}$ in $F$.

            In this paper we work with a non-trivial
            reparametrization of $\tau_G^{(b)}$, which reveals more
            remarkable properties of $b$-Hurwitz numbers. We will
            denote the reparametrized function
     by $\tau_G^{(\bb)}$, where the main parameter is $\bb =
     \sqrt{\alpha}^{-1}-\sqrt{\alpha}$ (see~\eqref{eq:TauBGWeight'}). We summarize the relation between the four equivalent parameters $\alpha, b, \mathsf k, \bb$ that will appear in our paper:
\begin{equation}
	\mathfrak{b} = \sqrt{\alpha}^{-1}-\sqrt{\alpha}=-\frac{b}{\sqrt{1+b}}= \mathsf k + r - 1\label{eq:parameters}.
      \end{equation}
      
      \noindent We assume that the rational functions $G(z)$ appearing throughout the paper are presented as irreducible fractions.

\subsection*{Acknowledgements}

NKC acknowledges the support of the ERC Starting Grant 948885, and the Royal Society University Research Fellowship.  This work is part of grant RYC2023-042878-I, funded by MCIN/AEI/10.13039/501100011033 and by the European Social Fund Plus (FSE+). This research was funded in whole or in part by {\it Narodowe Centrum Nauki},
 grant 2021/42/E/ST1/00162. For the purpose of Open Access, the authors have applied a CC-BY public copyright licence to any
Author Accepted Manuscript (AAM) version arising from this submission. KO is supported by JSPS KAKENHI Grant number 22KJ0715 and 23K12968.
We would like to thank the organizers of the program "TR SALENTO 2022", where our collaboration on this project begun. We would also like to thank Ga\"etan Borot and Thomas Creutzig for helpful discussions.

\section{Cut-and-join equation for weighted \texorpdfstring{$b$}{b}-Hurwitz numbers}

In this section, we will derive a specific cut-and-join equation for weighted $b$-Hurwitz numbers, which will turn out to be related to $\mathcal W$-algebra representations in the later sections.

\subsection{Jack polynomials}
\label{sub:Jack}

We call $\lambda$ an integer partition of $n$ if $\lambda$ is a finite
sequence
$\lambda_1 \geq \cdots \geq \lambda_\ell > 0$ of nonnegative integers
that sums up to $n =: |\lambda|$. We denote it by $\lambda \vdash n$, and
its number of parts $\ell$ is denoted $\ell(\lambda)$. Sometimes we use the notation $\lambda =
(1^{m_1(\lambda)},2^{m_2(\lambda)},\dots)$, where $m_i(\lambda)$
denotes the number of parts of $\lambda$ that are equal to $i$. The set of partitions of the same size is
linearly ordered:
\[ \lambda \leq \mu \iff \sum_{i=1}^j\lambda_i \leq \sum_{i=1}^j\mu_i
\qquad  \forall j.\]
We will
interchangeably use the terms \emph{an integer partition $\lambda$}
and \emph{a Young
diagram $\lambda$} and we associate the latter with the collection of boxes
$\{\lambda \ni \square := (x,y)\colon 1 \leq y \leq \ell(\lambda), 1
\leq x \leq \lambda_y\}$. 
For such a box $\square  = (x,y) \in \lambda$ we
define its $\alpha$-deformed content as the quantity $c_\alpha(x,y) :=
\alpha(x-1) - (y-1)$. This quantity plays an important role in theory of
Jack polynomials, that we will
briefly describe following~\cite{Stanley1989,Macdonald1995}.

Let $R$ be a ring and let $R(\alpha)[\pp]$ denote the polynomial ring in
infinitely many variables $\pp = (p_k)_{k \geq 1}$ over the ring $R(\alpha)$. The following operator, called the
\emph{Laplace--Beltrami operator}, acts on $R(\alpha)[\pp]$:
\begin{equation*}
  \label{eq:Laplace--BeltramiDef'}
  D_\alpha := \frac{1}{2}\left(\sum_{k,\ell \geq 1}\alpha
      k \ell \cdot p_{k+\ell}\partial_{p_{k}}\partial_{p_{\ell}}+\sum_{k,\ell \geq 1}
      (k+\ell)\cdot p_{k}p_{\ell}\partial_{p_{k+\ell}}+(\alpha-1)\sum_{k \geq 1}k(k-1)p_k \partial_{p_k}\right).
  \end{equation*}

\emph{Jack polynomials} $J^{(\alpha)}_\lambda$ are elements of
 $\Q(\alpha)[\pp]$ that are indexed by integer partitions
 $\lambda$. Consider the following grading on the ring $\Q(\alpha)[\pp]$
 defined on the basis: $\deg(\alpha) := 0, \deg(p_k) := k$. Then, the Jack
 polynomial $J^{(\alpha)}_\lambda$ can be characterized (up to a
 normalization factor) as the unique homogeneous polynomial of degree
 $|\lambda|$ such that:
 \begin{enumerate}
   \item  it is an eigenfunction of the Laplace--Beltrami operator: \label{eq:Laplace--BeltramiAction'} $D_\alpha J^{(\alpha)}_\lambda = \sum_{\square \in \lambda}c_\alpha (\square) J^{(\alpha)}_\lambda,$
  \item the transition matrix from $\{J^{\al}_\lambda\}_{|\lambda|=n}$
    to the monomial symmetric functions $\{m_\lambda(\xx)\}_{|\lambda|=n}$
    is lower triangular after identifying $p_k$ with the power-sum
    symmetric functions $p_k = p_k(\xx) := \sum_i x_i^k$.
   \end{enumerate}

We will work with the \emph{integral Jack polynomials} which are
normalized such that the coefficient
$[p_1^{|\lambda|}]J^{(\alpha)}_\lambda = 1$.
The quantity
\begin{equation*}
  \label{eq:JackNormalization}
  j^{(\alpha)}_\lambda := \prod_{\square \in
    \lambda}\hook_\lambda(\square)\cdot\hook_\lambda(\square)',
\end{equation*}
with
\[ \hook_\lambda(\square):= \alpha(\lambda_y-x)+\lambda_x'-y+1,\quad 
  \hook_\lambda(\square)' = \alpha(\lambda_y-x+1)+(\lambda_x'-y) \]
is the normalization factor.  One can define a scalar
product on $\Q(\alpha)[\pp]$ for which
the $J_\lambda^\al$ give an  orthogonal basis with the norm  $\|
J_\lambda^\al\|^2 = j_\lambda^\al$. One can show that this is
precisely the $\alpha$-deformed Hall scalar product for which $p_\lambda :=
\prod_{i=1}^{\ell(\lambda)}p_{\lambda_i}$ is orthogonal with the
classical norm:
\begin{equation*}
\langle p_\lambda, p_\la \rangle =
\alpha^{\ell(\lambda)}\frac{\prod_{i\geq 1}m_i(\la)!i^{m_i(\la)}}{|\la|!}.
\end{equation*}

Let us define the operators
$(\J^1_{k})_{k \in \Z}$ acting on $R[\tilde{\pp}]\llbracket \hbar \rrbracket$ by
\begin{equation}\label{eq:J1def}
  \J^1_{k} = \begin{cases} \hbar k\cdot \partial_{\tilde{p}_{k}}
    &\text{ if } k> 0,\\
    0 &\text{ if } k= 0,\\
    \hbar \tilde{p}_{-k} &\text{ if } k < 0.
  \end{cases}
\end{equation}
Since $[\J^1_{k}, \J^1_{\ell}] = k\hbar^2\delta_{k+\ell,0}$, the
action of
$(\J^1_{k})_{k\in \Z}$ on $R[\tilde{\pp}]\llbracket \hbar \rrbracket$ is a representation of
the Heisenberg algebra. Furthermore, with the change of variables $\tilde{p}_k=
\sqrt{\alpha}^{-1} p_k$ one has $\partial_{\tilde{p}_k} =
\sqrt{\alpha}\partial_{p_k}$, hence $J^\al_\lambda$ can be
considered as an element of $\Q(\sqrt{\alpha})[\tilde{\pp}]$ and the
rescaled Laplace--Beltrami operator $\tilde{D}_\alpha :=
\frac{\hbar}{\sqrt{\alpha}} D_\alpha$ can be written as
\begin{equation}
  \label{eq:Laplace--BeltramiDef}
  \tilde{D}_\alpha := \frac{\hbar^{-2}}{2}\sum_{k\geq1}\J^1_{-k}\left(\sum_{\ell\geq1}\J^1_{\ell}  \J^1_{k-\ell}
   -(k-1)\hbar\cdot\bb\cdot\J^1_{k}\right),
\end{equation}
and 
 \begin{equation}
   \label{eq:Laplace--BeltramiAction}
    \tilde{D}_\alpha J^{(\alpha)}_\lambda = \hbar \sum_{\square \in \lambda}\tilde{c}_\alpha (\square) J^{(\alpha)}_\lambda,
  \end{equation}
  where $\tilde{c}_\alpha (x,y) :=
  \sqrt{\alpha}(x-1)-\sqrt{\alpha}^{-1}(y-1)$, and $\bb$ is given by \eqref{eq:parameters}.

           \subsection{Nazarov--Sklyanin operators and co-transition measure}

Consider the (infinite) row vector
$ \J^1_+= (\J^1_k)_{{k} \in \Z_{\geq 1}}$ and dually the column vector $\J^1_- =
(\J^1_{-k})_{{k} \in \Z_{\geq 1}}$. Let $L = (L_{k,\ell})_{k,\ell \in
  \Z_{\geq 1}}$ be the infinite matrix defined by
$L_{k,\ell} := \J^1_{k-\ell}-\delta_{k,\ell}\,k\hbar\bb$, for all $k,\ell\in\Z_{\geq 1}$:
$$
L = \begin{bmatrix} 
-\hbar\bb & \J^1_{-1} & \J^1_{-2} & \cdots\\
\J^1_{1} & -2\hbar\bb & \J^1_{-1} & \ddots\\
\J^1_{2} &\J^1_{1} & -3\hbar\bb & \ddots\\
\vdots & \ddots & \ddots & \ddots 
\end{bmatrix}.
$$

Define the following generating series:
\begin{equation}
  \label{eq:defBoolean}
  \sum_{\ell \geq 0}B^{(\alpha)}_{\ell+2}(\lambda)z^{-\ell-1} = z - (z+\sqrt{\alpha}^{-1}\ell(\lambda))\prod_{i=1}^{\ell(\lambda)}\frac{z+\sqrt{\alpha}^{-1} (i-1)-\sqrt{\alpha}\lambda_i}{z+\sqrt{\alpha}^{-1} i-\sqrt{\alpha}\lambda_i}.
\end{equation}

The main
result of Nazarov--Sklyanin~\cite[Theorem 2]{NazarovSklyanin2013}
is a construction of a family of commuting operators that includes the
Euler and Laplace--Beltrami operators, whose eigenfunctions are Jack
polynomials with explicit eigenvalues:

\begin{theorem}
  \label{theo:Nazarov-Sklyanin}
  The following equality holds true for all $\ell\geq 0$:
  \begin{align*}
        \left(\J_- L^\ell \J_+\right)  J^{(\alpha)}_\lambda &=
    \hbar^{\ell+2}B^{(\alpha)}_{\ell+2}(\lambda) \cdot J^{\al}_\lambda. \label{eq:NS-Boolean}
    \end{align*}
  \end{theorem}

  Note that the operators $\hbar^{-2} \J_- \J_+$ and $\frac{\hbar^{-2}}{2}\J_- (L+\hbar\bb\Id) \J_+$ are precisely
  the Euler and rescaled Laplace--Beltrami operators, and
  \eqref{eq:Laplace--BeltramiAction} is a special case of \cref{theo:Nazarov-Sklyanin}.

The coefficients $B^{(\alpha)}_{\ell+2}(\lambda)$ can be interpreted
as the Boolean cumulants of a transition measure of
Kerov, or equivalently as the moments of a
co-transition measure of Kerov, both introduced
in~\cite{Kerov1993} to study random Young diagrams. It appears that this fact
was not known to Nazarov and Sklyanin, and a connection
to probability was
first mentioned in~\cite{Moll2015}. Later this connection was
described in greater generality and Nazarov--Sklyanin operators were
used to prove the LLN and CLT for Jack measures~\cite{Moll2023,CuencaDolegaMoll2026}, and Lassalle's
positivity conjecture~\cite{BenDaliDolega2023}. We will  use this
connection here to find an explicit formula for the action of $\J_-
(L+\hbar\bb\Id)^\ell \J_+$ on Jack polynomials that will be crucial in
the next section.

Define a
probability measure $\nu^{(\alpha)}_{\lambda}$, called a
\emph{co-transition measure}, as the measure uniquely characterized by
its Cauchy transform expanded around infinity:
\begin{equation}
  \label{eq:Cauchy-St}
  \sqrt{\alpha}|\lambda|\int_\R \frac{d\nu^{(\alpha)}_\lambda(x)}{z-x} =z - (z+\sqrt{\alpha}^{-1}\ell(\lambda))\prod_{i=1}^{\ell(\lambda)}\frac{z+\sqrt{\alpha}^{-1} (i-1)-\sqrt{\alpha}\lambda_i}{z+\sqrt{\alpha}^{-1} i-\sqrt{\alpha}\lambda_i}.
\end{equation}
The fact that $\nu^{(\alpha)}_\lambda $ is a probability measure is due to
Kerov~\cite[Lemma (2.6)]{Kerov2000}. Note that $\nu^{(\alpha)}_\lambda $
is supported on a discrete set located at the
$\alpha$-contents of the outer corners of
$\lambda$ that can be removed to obtain a smaller diagram $\mu$:
\[y \in \supp(\nu^{(\alpha)}_\lambda) \iff \mu \subset \lambda, |\mu|
  = |\lambda|-1, y = \big(\tilde{c}_\alpha(\lambda\setminus\mu)-\bb\big).\]
Therefore
\begin{equation}
  \label{eq:CoTransitionForm}
  \nu^{(\alpha)}_\lambda = \sum_{\mu
               \nearrow
               \lambda}f_\mu^\lambda\cdot\delta_{\big(\tilde{c}_\alpha(\lambda\setminus\mu)-\bb\big)},
\end{equation}
where $\mu
               \nearrow
               \lambda$ means that we sum over all Young diagrams $\mu$
               that are obtained from $\la$ by removing a box. It was
               proven by Kerov~\cite[Theorem (6.7)]{Kerov2000} that
               the masses $f_\mu^\lambda$ can be computed as:
               \begin{equation}
                 \label{eq:KerovMasses}
                 f_\mu^\la = \frac{\langle \J^1_{-1} J^\al_\mu, J^\al_\la \rangle}{\hbar\sqrt{\alpha}|\lambda|j_\mu^\al}.
               \end{equation}

We will use the relation between Nazarov--Sklyanin operators and the co-transition
measure to prove the following lemma that will be crucial in the next
subsection.

\begin{lemma}
    \label{cor:Boolean}
  Let $\la \vdash n+1$. The following identity holds true:
  \[ \big(\J_-(L+\hbar\bb\Id)^\ell \J_+\big)
    J^{(\alpha)}_{\la} = \frac{\hbar}{\sqrt{\alpha}}\sum_{\mu \vdash n}\frac{\langle \J^1_{-1}
    J_\mu^\al, J_\la^\al \rangle}{ j^{(\alpha)}_\mu}
    \left(\hbar \tilde{c}_\alpha(\la\setminus\mu)\right)^\ell J^{(\alpha)}_{\la}.
            \]
         \end{lemma}

         \begin{proof}
           Fix $\la \vdash n+1$. It is well-known
           (see~\cite{Stanley1989}) that $\langle \J^1_{-1} J_\mu^\al,
           J_\la^\al \rangle$ is nonzero only if $\mu \nearrow
           \la$. Therefore \eqref{eq:CoTransitionForm}
           and \eqref{eq:KerovMasses} imply that
           \[ \hbar \sqrt{\alpha} \cdot|\la|\cdot \nu_\la^\al = \sum_{\mu \vdash
               n}\langle \J^1_{-1} J_\mu^\al, J_\la^\al \rangle
             \frac{\delta_{\big(\tilde{c}_\alpha(\lambda\setminus\mu)-\bb\big)}}{j_\mu^\al}.\]
           This formula, together with \eqref{eq:defBoolean} and
           \eqref{eq:Cauchy-St}, immediately gives that for all $\ell
           \geq 0$
           \[ B^\al_{\ell+2}(\la) = \sum_{\mu \vdash
               n} \frac{\langle \J^1_{-1} J_\mu^\al, J_\la^\al \rangle}{\hbar\sqrt{\alpha} j_\mu^\al}
             \left(\tilde{c}_\alpha(\la\setminus\mu)-\bb\right)^\ell.\]
           We conclude by applying \cref{theo:Nazarov-Sklyanin}.
       \end{proof}

       \subsection{Weighted $b$-Hurwitz numbers and cut-and-join
         equation}
       \label{sub:bHurwitz}

Following~\cite{ChapuyDolega2022} we define the generating function of
$G$-weighted $b$-Hurwitz numbers $\HurGbbg (\mu_1,\dots,\mu_\ell)$ by
the following identity:
\begin{multline}
  \label{eq:TauBGWeight'}
\tau_G^{(\bb)} := \sum_{n \geq 0} \left(\frac{\sqrt{\alpha}}{\hbar} \right)^n\sum_{\lambda \vdash n} 
		\frac{J_\lambda^{(\alpha)}(\sqrt{\alpha}\tilde{\pp})}{j_\lambda^{(\alpha)}}\prod_{\square
                  \in \lambda}G(\hbar\cdot \tilde{c}_{\alpha}(\square))=\\
                =\exp\left(\sum_{g
          \in
          \frac{1}{2}\cdot\Z_{\geq 0}}\sum_{n \in \Z_{\geq
            1}}\hbar^{2g-2+n}\sum_{\mu \colon \ell(\mu)
          = n}\HurGbbg(\mu_1,\dots,\mu_n)
          \cdot \tilde{p}_\mu\right),
    \end{multline}
    where $G(z) = \sum_{i=0}^\infty g_i z^i$. These are so-called
    \emph{single} $G$-weighted $b$-Hurwitz numbers, and the
    $\tau_G^{(\bb)}$ given by \eqref{eq:TauBGWeight'} 
    corresponds to the specialization $\tau_G^{(b)}(\pp,\qq)\big|_{q_k
      = \delta_{k,1}}$ of the generating function of double
    $G$-weighted $b$-Hurwitz numbers introduced in
    \cite{ChapuyDolega2022} after the reparametrization
    \[ \tilde{p}_k = \sqrt{\alpha}^{-1}p_k,\ \ \hbar =
      \sqrt{\alpha},\ \  \bb =- \sqrt{\alpha}^{-1} b.\]
            
    It is convenient to add a redundant
    parameter $t$ by the substitution $\tau_G^{(\bb)}(t,\hbar,\tilde{\pp}) := \tau_G^{(\bb)}\big|_{\tilde{p}_k
      = t^k\tilde{p}_k}$. It is not hard to show that
    \[\hbar t\partial_t\log\tau_G^{(\bb)}(t,\hbar,\tilde{\pp}) \in
      \Q(\alpha)[\sqrt{\alpha}\tilde{\pp},\g,\hbar\sqrt{\alpha}^{-1}]\llbracket t \rrbracket.\]
    However, a much stronger result due to Chapuy and the second author holds, which explains the meaning of the parameter
    $\sqrt{\alpha}\cdot \bb$:

    \begin{theorem}[\cite{ChapuyDolega2022}]
      \label{theo:CD22}
      For any partition $\mu$ the $G$-weighted $b$-Hurwitz number
      \[\alpha^g\cdot|\mu|\HurGbbg(\mu_1,\dots,\mu_n)
        \in \Z_{\geq 0}[-\sqrt{\alpha}\cdot\bb][\g]\]
      is a weighted
      generating series of labeled real meromorphic functions $f$
      of degree $|\mu|$ with ramification profile
      $(\mu_1,\dots,\mu_n)$ over $\infty$, where the weight is an
      explicit expression in $\g$ and $\sqrt{\alpha}\bb$ that depends
      on the combinatorial data associated with $f$.
    \end{theorem}

    \cref{theo:CD22} is the abbreviated version of the main result
    of~\cite{ChapuyDolega2022}, and it is stated here to
    highlight topological/combinatorial relevance of the $G$-weighted
    $b$-Hurwitz numbers; we will not use it here and therefore  refer to~\cite{ChapuyDolega2022} for the relevant
    definitions and  details.

    The reparametrization introduced here is motivated by our main
    result that gives an explicit relation between the reparametrized
    function $\tau^{(\bb)}_G$ and the Airy structure coming from
    representations of $\mathcal{W}$-algebras. In particular, one can notice by comparing
    \eqref{eq:TauBGWeight'} with the parametrization used in
    \cref{sub:b-Hur} that $\HurGbg(\mu_1,\dots,\mu_n)=
    \alpha^g\cdot\HurGbbg(\mu_1,\dots,\mu_n)$, and we will show that
    $|\mu|\tilde{H}_{G;g}^{(\bb)}(\mu_1,\dots,\mu_n) \in \Z_{\geq
      0}[-\bb,\g]$, which eliminates the dependence on $\alpha$ and is stronger than \cref{theo:CD22}. In order to do this, we intend
    to find an explicit differential operator of
    finite degree that annihilates $\tau_G^{(\bb)}$ and uniquely
    determines it by the initial condition
    $\tau_G^{(\bb)}\big|_{t=0}=1$ --- this is why we introduced the
    parameter $t$.

    \begin{remark}
      \label{rem:RationalRelation}
      In the following we will mostly work with $G$ being a rational
      function. A generic rational function can be
      parametrized either by $G_u(z) =
      u\frac{\prod_{i=1}^n(P_i+z)}{\prod_{j=1}^m(Q_j-z)}$ for some
      $u,\PP,\QQ$, or by $\check G_u(z) =
      u\frac{\prod_{i=1}^n(1+P_iz)}{\prod_{i=1}^m(1-Q_jz)}$. Note that
      the change of variables $u \mapsto \frac{P_1\cdots
        P_n}{Q_1\cdots Q_m}u, P_i \mapsto P_i^{-1}, Q_j^{-1} \mapsto
      Q_j$ for all $i \in [n]$, $j \in [m]$ transform $G_u$ into
      $\check G_u$. Moreover, $u$ can be absorbed into the parameter $t$ of
      $\tau^{(\bb)}$ by
      substituting $t \mapsto ut$, so that we can equivalently work
      with $G(z) =
      \frac{\prod_{i=1}^n(P_i+z)}{\prod_{j=1}^m(Q_j-z)}$ or $\check G(z) =
      \frac{\prod_{i=1}^n(1+P_iz)}{\prod_{j=1}^m(1-Q_jz)}$, as their generating functions are 
     related by 
     \begin{equation}
       \label{eq:RationalRelation}
       \tau^{(\bb)}_{G}(t,\tilde{\pp},\PP,\QQ,\hbar) =
        \tau^{(\bb)}_{\check G}\left( t\cdot \frac{P_1\cdots
          P_n}{Q_1\cdots
          Q_m},\tilde{\pp},\PP^{-1},\QQ^{-1},\hbar\right).
      \end{equation}
Note that $\check G(z)$ considered as a formal power series has positive
integer coefficients $\check G(z) \in \Z_{\geq 0}[\PP,\QQ]$ therefore it is
a more natural parametrization from an enumerative point
 of view as \cref{theo:CD22} implies that $\alpha^g\cdot|\mu|H_{\check G;g}^{(\bb)}(\mu_1,\dots,\mu_n)
        \in \Z_{\geq 0}[-\sqrt{\alpha}\cdot\bb,\PP,\QQ]$ and each
        monomial in $-\sqrt{\alpha}\cdot\bb, \PP,\QQ$ is counting
        certain generalized branched coverings. However, the
        parametrization $G(z)$ is suited much better for our
        purpose of connecting $\tau^{(\bb)}_{G}$ with
        $\mathcal{W}$-algebras, so we will mostly assume the weight function to be of the form $G= \frac{\prod_{i=1}^n(P_i+z)}{\prod_{j=1}^m(Q_j-z)}$.
         \end{remark}

\subsubsection{Weighted lattice paths}\label{sec:weightedpaths}

   Sometimes, operators annihilating $\tau^{(\bb)}_G$ can be expressed in a particularly elegant
    form involving weighted lattice paths. This point of view was
    applied in~\cite{Moll2023,CuencaDolegaMoll2026} to
    Nazarov--Sklyanin operators and it turned out to be very useful in
    studying the structure of Jack polynomials~\cite{BenDaliDolega2023}. In the
following we further develop this approach to connect $G$-weighted
$b$-Hurwitz numbers with $\mathcal{W}$-algebras.

All the paths we consider will be directed paths on the lattice
$\mathbb Z^2$ where integral points $(x,y) \in \Z^2$ are connected by steps of the form $(1,k)$ with $k
\in \mathbb Z$. We will refer to the first coordinate of the integral
points as the
$X$-coordinate and the second coordinate as the $Y$-coordinate. We only consider such paths and hence we will simply refer to them as paths. 

\begin{definition}
	Suppose that $x<x'$ and $y,y'$ are integers. We define  $\Gamma_{(x,y) \to (x',y')}$ to be 
	the set of \emph{paths} starting from $(x,y)$ and finishing at
	$(x',y')$, and we say that a path $\gamma \in \Gamma_{(x,y) \to (x',y')}$ has length $x'-x$,
        so that it contains $x'-x$ steps.
\end{definition}

Consider three integers $x < x' < x''$. Given two paths, say $\gamma \in \Gamma_{(x,y) \to (x',y')}$ and
$\gamma' \in \Gamma_{(x',y') \to (x'',y'')}$,  we can
construct a new path $\gamma \cup \gamma' \in \Gamma_{(x,y) \to
	(x'',y'')}$ by concatenation.  We also consider special types of paths, called bridges, that stay  above the $X$-axis.
\begin{definition}\label{def:bridges} 	Suppose that $x<x'$ and $y,y'$ are integers. We define a \emph{bridge}  to be a  path $\gamma \in \Gamma_{(x,y) \to
		(x',y')}$ such that all the integral points of
              $\gamma$ (except the points $(x,y)$ and $(x',y')$) have non-negative $Y$-coordinate. We denote the set of bridges between $(x,y)$ and $(x',y')$ by $\Gamma^{\geq 0}_{(x,y) \to
		(x',y')}$. 
            \end{definition}
            In the following, we will assign various operator-valued
            weights to a path $\gamma$. For a
            family of indeterminates $\uu := (u_1,\dots,u_r)$ and for
            a path $\gamma$ of length $r$ we assign a weight
         $\widetilde{\wt}(\gamma\,|\uu) :=\widetilde{\wt} (s_1|\uu)\cdots \widetilde{\wt} (s_{r}|\uu)$, where $s_i$ is
         the $i$-th step of $\gamma$ and
         \begin{equation}\label{eq:tildewt}
         	\widetilde{\wt} (s_i|\uu)
         = \begin{cases} \J^1_{-k} &\text{ if } s_i = (1,k), k\neq
           0,\\ u_{i}-\hbar \bb \ell &\text{ if }s_i = (1,0),
           \text{ and connects } (i-1,\ell) \text{ with }
           (i,\ell).\end{cases}
         \end{equation}

       \begin{theorem}\label{thm:rationalweight}
  Let $G(z) =
\frac{\prod_{i=1}^n\left(P_i+z\right)}{\prod_{i=1}^m\left(Q_i-z\right)}$. The
generating function $\tau^{(\bb)}_G$ defined by \eqref{eq:TauBGWeight'} is the unique formal power series
$\tau \in \Q(\QQ)[\PP, \tilde{\pp}, \bb](\!( \hbar)\!)\llbracket
t\rrbracket$ which satisfies the following PDE
  \begin{equation}
    \label{eq:CutAndJoinGeneralRational}
    t\cdot \sum_{\gamma \in \Gamma^{\geq 0}_{(0,-1) \to
          (n+1,0)}}\widetilde{\wt}(\gamma\,|(0,\PP))\tau =(-1)^m\sum_{\gamma \in \Gamma^{\geq 0}_{(0,-1) \to
          (m+2,-1)}}\widetilde{\wt}(\gamma\,|(0,-\QQ,0))\tau\end{equation}
  with the initial condition $\tau = 1 + O(t)$, where $(0,\PP) =
  (0,P_1,\dots, P_n)$ and $(0,-\QQ,0) =
  (0,-Q_1,\dots, -Q_m,0)$.
\end{theorem}

\noindent We will deduce \cref{thm:rationalweight} from a more general result.

         \begin{theorem}
           \label{theo:GeneralCutJoin}
Suppose that there exist invertible formal power series $F(z) :=
\sum_{i\geq 0}f_i
z^i$, and $H(z) :=
\sum_{i\geq 0}h_i
z^i$ such that
\begin{equation}
  \label{eq:EquationForG}
  G(z) := \sum_{i\geq 0} g_i z^i =  H(z)\cdot F(z)^{-1}.
  \end{equation}
The
generating function $\tau^{(\bb)}_G$ defined by \eqref{eq:TauBGWeight'} is the unique formal power series
$\tau \in \Q[\tilde{\pp}, f_0^{-1},f_1,f_2,\dots,\hh,\bb](\!( \hbar)\!)\llbracket
t\rrbracket$ which satisfies the following PDE
\begin{equation}
  \label{eq:GeneralCutJoin}
  t \sum_{i \geq 0}h_i\sum_{\gamma \in \Gamma^{\geq 0}_{(0,-1) \to
          (i+1,0)}}\widetilde{\wt}(\gamma\,|0)\tau = \sum_{i \geq
     0}f_{i}\sum_{\gamma \in \Gamma^{\geq 0}_{(0,-1) \to
          (i+2,-1)}}\widetilde{\wt}(\gamma\,|0)\tau.
    \end{equation}
      with the initial condition $\tau = 1 + O(t)$. Moreover
      $|\mu|\HurGbbg(\mu_1,\dots,\mu_n) \in \Z_{\geq
        0}[-\bb][\g]$ is
      a polynomial in $-\bb,\g$  with
      non-negative integer coefficients. In addition, when $2g$ is odd (resp. even), the  non-zero coefficients of $|\mu|\HurGbbg(\mu_1,\ldots,\mu_n) $ have odd (resp. even) powers of $ -\bb$.
    \end{theorem}

    \begin{remark}
      Note that one can always pick $H(z) = G(z)$, and $F(z)=1$. For
      this choice, \eqref{eq:GeneralCutJoin} recovers the cut-and-join
      equation \cite[Eq. (62)]{ChapuyDolega2022}.
      \end{remark}

           Recall that $\Ad_A(B) := [A,B]$ is the
           adjoint operator and that $\tilde{D}_\alpha$ is the
           Laplace--Beltrami operator given by
           \eqref{eq:Laplace--BeltramiDef}. Before we prove \cref{theo:GeneralCutJoin} we need the
           following lemma.

           \begin{lemma}
             \label{lem:AdjointPaths}
             For any $\ell \geq 0$ one has:
             \begin{equation*}
               \label{eq:adjoin}
               \Ad^{i}_{\tilde{D}_{\alpha}}(\J^1_{-1}) = \sum_{\gamma \in \Gamma^{\geq 0}_{(0,-1) \to
          (i+1,0)}}\widetilde{\wt}(\gamma\,|0).
               \end{equation*}
             \end{lemma}

             \begin{proof}
               Chapuy and the second author proved
               in~\cite[Theorem 4.7]{ChapuyDolega2022} that
               \[\Ad^{i}_{\tilde{D}_{\alpha}}(\J^1_{-1})  =
                 \sum_{k\geq
                   0}\J^1_{-k-1}\partial_{y_k}\Lambda_{Y}^{i} y_0, \]
               where
\[
\Lambda_{Y} := \sum_{k,\ell\geq 1}y_{k+\ell-1}\partial_{y_{\ell-1}}\J^1_k +\sum_{k,\ell\geq
  1}y_{\ell-1}\partial_{y_{k+\ell-1}}\J^1_{-k}
-\hbar\bb\cdot\sum_{k\geq
  0}k\cdot y_{k}\partial_{y_k}.\]
Therefore it is enough to prove that
\[ \Lambda_{Y}^{i} y_0 = \sum_{k\geq
                   0}y_{k}\sum_{\gamma \in \Gamma^{\geq 0}_{(1,k) \to
                     (i+1,0)}}\widetilde{\wt}(\gamma\,|0)\]
               with the convention that $\Gamma_{(x,y) \to
                     (x,y)} = \Gamma^{\geq 0}_{(x,y) \to
                     (x,y)} = \{(x,y)\}$, and
                   $\widetilde{\wt}((x,y)\,|0) = 1$. This identity
                   follows easily by induction on $i$ by noticing that
                   \[\Lambda_{Y} y_\ell = \sum_{k \geq 0}y_k \sum_{\gamma \in \Gamma^{\geq 0}_{(1,k) \to
                     (2,\ell)}}\widetilde{\wt}(\gamma\,|0). \qedhere\]
               \end{proof}

               \begin{proof}[Proof of \cref{theo:GeneralCutJoin}]
                 In order to make the notation lighter we write $\tau =
        \tau^{(\bb)}_G$, $J_\lambda=J_\lambda^{\al}(\sqrt{\alpha}\tilde{\pp})$, $
        j_\lambda=j_\lambda^{\al}$, $c(\square) = \tilde{c}_{\alpha}(\square)$, and we define
	\[ \tilde{J}_\lambda := \frac{J_\la}{j_\la}\left(\frac{\sqrt{\alpha}}{\hbar}\right)^{|\lambda|}\prod_{\square
            \in \lambda}G(\hbar \tilde{c}(\square)).\]
 Notice that
 \begin{multline*}
   \label{eq:Multp1'}
   h_0 \cdot \J^1_{-1}\tilde{J}_\lambda =
   h_0\cdot \frac{\hbar}{\sqrt{\alpha}} \sum_{\la
     \nearrow \mu}\frac{\langle \J^1_{-1} J_\lambda, J_\mu\rangle}{j_\la} G(\hbar \tilde{c}(\mu\setminus\lambda))^{-1}\tilde{J}_\mu \\
   = \frac{\hbar}{\sqrt{\alpha}} \sum_{\la \nearrow \mu}\frac{\langle \J^1_{-1} J_\lambda, J_\mu\rangle}{j_\la}\left(\sum_{i \geq
     0}f_i (\hbar \tilde{c}(\mu\setminus\lambda))^{i}-G(\hbar \tilde{c}(\mu\setminus\lambda))^{-1}\sum_{i \geq
     1}h_i (\hbar \tilde{c}(\mu\setminus\lambda))^{i}\right)
   \tilde{J}_\mu
   \end{multline*}
        by our assumption \eqref{eq:EquationForG}. Therefore
       \[ [t^{n+1}]t\cdot h_0 \cdot   \J^1_{-1}\cdot\tau\]
       can be decomposed as a difference of two terms:
\[A_1 = \frac{\hbar}{\sqrt{\alpha}} \sum_{\lambda \vdash n}\sum_{\mu \vdash n+1}\bigg(\sum_{i \geq
     0}f_i (\hbar \tilde{c}(\mu\setminus\lambda))^{i} \bigg)\frac{\langle \J^1_{-1}
      J_\lambda, J_\mu\rangle}{j_\la} \tilde{J}_\mu=\bigg(\sum_{i \geq
     0}f_i\J_-(L+\hbar\bb\Id)^i \J_+ \bigg)[t^{n+1}]\tau\]
by \cref{cor:Boolean}, and
\[         A_2 = \frac{\hbar}{\sqrt{\alpha}}\sum_{\lambda \vdash n}\sum_{\mu \vdash n+1}\sum_{i \geq
     1}h_i(\hbar \tilde{c}(\mu\setminus\lambda))^{i} G(\hbar \tilde{c}(\mu\setminus\lambda))^{-1}\frac{\langle \J^1_{-1}
      J_\lambda, J_\mu\rangle}{ j_\la}
               \tilde{J}_\mu.\]
             We claim that
             \[ A_2 = \bigg(\sum_{i \geq
                 1}h_{i}\Ad^{i}_{\tilde{D}_\alpha}(\J^1_{-1})\bigg)[t^{n}]\tau,\]
             which is equivalent to proving that for each $i \geq 1$
             one has
             \[ \sum_{\lambda \vdash n}\Ad^{i}_{\tilde{D}_\alpha}(\J^1_{-1})
               \tilde{J}_\lambda = \frac{\hbar}{\sqrt{\alpha}}\sum_{\lambda \vdash n}\sum_{\mu
                 \vdash n+1}(\hbar \tilde{c}(\mu\setminus\lambda))^{i} G(\hbar \tilde{c}(\mu\setminus\lambda))^{-1}\frac{\langle \J^1_{-1}
      J_\lambda, J_\mu\rangle}{ j_\la}
    \tilde{J}_\mu.\]
  For
  $i = 0$ the above equality is immediate from the definition, and for $i \geq 1$ we will proceed by
  induction: 
  \begin{multline*}
    \sum_{\lambda \vdash n}\Ad^{i}_{\tilde{D}_\alpha}(\J^1_{-1}) \tilde{J}_\lambda
               = \frac{\hbar}{\sqrt{\alpha}}\sum_{\lambda \vdash n}\sum_{\mu \vdash
                 n+1}(\hbar\tilde{c}(\mu\setminus\lambda))^{i-1}
               G(\hbar \tilde{c}(\mu\setminus\lambda))^{-1}\frac{\langle \J^1_{-1}
      J_\lambda, J_\mu\rangle}{ j_\la}\cdot\\
               \cdot\left(\tilde{D}_\alpha\tilde{J}_\mu-\hbar\sum_{\square \in
                   \lambda}\tilde{c}(\square)\tilde{J}_\mu\right) = \frac{\hbar}{\sqrt{\alpha}}\sum_{\lambda \vdash
                 n}\sum_{\mu \vdash n+1}(\hbar
               \tilde{c}(\mu\setminus\lambda))^{i} G(\hbar \tilde{c}(\mu\setminus\lambda))^{-1}\frac{\langle \J^1_{-1}
      J_\lambda, J_\mu\rangle}{ j_\la}
    \tilde{J}_\mu,
    \end{multline*}
    where we have used \eqref{eq:Laplace--BeltramiAction}.
    Finally, \cref{lem:AdjointPaths} implies that
    \[A_2 = \sum_{i=1}^\infty h_i\sum_{\gamma \in \Gamma^{\geq 0}_{(0,-1) \to
          (i+1,0)}}\widetilde{\wt}(\gamma\,|0)[t^n]\tau\]
    and the fact
    that
    \[A_1 = \sum_{i=0}^\infty f_i\sum_{\gamma \in \Gamma^{\geq 0}_{(0,-1) \to
          (i+2,-1)}}\widetilde{\wt}(\gamma\,|0)[t^{n+1}]\tau\]
    follows
      easily from the interpretation of the operator $\J_-(L+\hbar\bb\Id)^i
      \J_+$ as a sum over paths -- see for instance \cite[Theorem
      (3.10)]{CuencaDolegaMoll2026}.

             To prove uniqueness, suppose  that there are
             two different solutions $\tau_1$, and
             $\tau_2$
             of~\eqref{eq:GeneralCutJoin}. Then,
             $\tau' := \tau_1 - \tau_2 = O(t)$
             also satisfies~\eqref{eq:GeneralCutJoin}. Let
             $k \geq 1$ be the smallest $k$ such that $[t^k]\tau'
             \neq 0$. Our assumption on $k$ guarantees that the LHS
of~\eqref{eq:GeneralCutJoin} acting on $\tau'$ gives a
formal power series which is $O(t^{k+1})$. On the other hand,
extracting the coefficient $[t^k]$ of the action of the RHS
of~\eqref{eq:GeneralCutJoin} on $\tau'$ is the same as the action of
the RHS of~\eqref{eq:GeneralCutJoin} on $[t^k]\tau'$. One can introduce a grading on $\Q[\tilde{\pp}, f_0^{-1},f_1,\dots,\hh,\bb](\!( \hbar)\!)\llbracket
t\rrbracket$ by giving the only non-zero grading to $\deg(f_i) = \deg(h_i) = 1$ for $i \in
\Z_{\geq 1}$. Then the RHS of~\eqref{eq:GeneralCutJoin} can be written as
\[ f_0\sum_{k\geq1} \J^1_{-k}\J^1_{k} + \text{ terms of degree
    greater or equal to } 1.\]
Note that $\sum_{k\geq1} \J^1_{-k}\J^1_{k} \tilde{p}_\mu = \hbar^2|\mu|\tilde{p}_\mu$ for
any partition $\mu$. In particular let $A_k \neq 0$ denote the smallest
degree term of $[t^k]\tau'$. Then the action of the RHS on
$[t^k]\tau'$ is equal to
\[ k f_0\hbar^2\cdot A_k + \text{higher degree terms} \neq 0,\]
which contradicts that this power series is $O(t^{k+1})$. This
finishes the proof of the uniqueness.

Finally, \eqref{eq:GeneralCutJoin} implies (by taking $H(z) = G(z),
F(z)=1$) that $|\mu|\HurGbbg(\mu_1,\dots,\mu_n) \in
\Q[\bb][\g]$. But \cref{theo:CD22} says that $\sqrt{\alpha}^{2g}\cdot|\mu|\HurGbbg(\mu_1,\dots,\mu_n) \in
\Z_{\geq 0}[-\bb\cdot \sqrt{\alpha}][\g] = \Z_{\geq 0}[\alpha-1][\g]$,
therefore $|\mu|\HurGbbg(\mu_1,\dots,\mu_n) \in
\Q[\bb][\g]$ is a polynomial of the same parity as $2g$. The
fact that $\Q[\bb]$ could be replaced by $\Z_{\geq 0}[-\bb]$ follows easily by comparing
the coefficients of $|\mu|\HurGbbg(\mu_1,\dots,\mu_n)$ with the
coefficients of $\alpha^g |\mu|\HurGbbg(\mu_1,\dots,\mu_n)$.
\end{proof}

\begin{remark}
In the case of a rational weight of the form $\check G(z) :=
\frac{\prod_{i=1}^n\left(1+P_iz\right)}{\prod_{i=1}^{m}\left(1-Q_i
    z\right)}$ we
can prove directly that $|\mu|H_{\check G}^{(\bb)}(\mu_1,\dots,\mu_n) \in
\Z_{\geq 0}[-\bb][\PP,\QQ]$ without referring to
\cite{ChapuyDolega2022}. Indeed, note that the latter statement is
equivalent to showing that $\hbar t\partial_t \log(\tau^{(\bb)}_{\check G}) \in
\mathbb{Z}_{\geq 0}[-\bb][\PP,\QQ,\tilde{\pp},\hbar]\llbracket t \rrbracket$. Note however
that $\hbar t\partial_t \log(\tau^{(\bb)}_{\check G}) = \sum_{k=1}^\infty
\tilde{p}_k \J^1_k \log(\tau^{(\bb)}_{\check G})$, and we can prove a stronger
statement that $\J^1_k \log(\tau^{(\bb)}_{\check G}) \in \mathbb{Z}_{\geq 0}
[-\bb][\PP,\QQ,\tilde{\pp},\hbar]\llbracket t \rrbracket$. This
statement can be proven using \eqref{eq:splitcutandjoin}, which
refines
\eqref{eq:CutAndJoinGeneralRational}, and then making an appropriate
change of variables given by \eqref{eq:RationalRelation}.  Equation \eqref{eq:splitcutandjoin}
produces a system of partial differential equations involving
$\{\J^1_k \log(\tau^{(\bb)}_{\check G})\}_{k \geq 1}$ that allows one to compute
$\J^1_k \log(\tau^{(\bb)}_{\check G})$ inductively order by order w.r.t
$t$, and since it involves sums of terms that were inductively assumed to lie
in $\mathbb{Z}_{\geq 0}
[-\bb][\PP,\QQ,\tilde{\pp},\hbar]\llbracket t \rrbracket$, the statement
follows. It is not our intention to focus on this alternative proof,
so we leave the details to the interested reader.
\end{remark}

\begin{proof}[Proof of \cref{thm:rationalweight}]
  By taking $H(z) = \prod_{i=1}^n
            (P_i+z)$, $F(z) = \prod_{i=1}^m(Q_i-z)$ and applying
            \cref{theo:GeneralCutJoin} we have that $\tau^{(\bb)}_G$ satisfies
            the following equation
\[t\sum_{i=0}^n e_{n-i}(\PP)\sum_{\gamma \in \Gamma^{\geq 0}_{(0,-1) \to
                  (i+1,0)}}\widetilde{\wt}(\gamma\,|0)\tau^{(\bb)}_G =(-1)^m\sum_{i =0}^m e_{m-i}(-\QQ) \sum_{\gamma \in \Gamma^{\geq 0}_{(0,-1) \to
                  (i+2,-1)}}\widetilde{\wt}(\gamma\,|0)\tau^{(\bb)}_G,\]
  where $e_i(\xx)$ is the $i$-th elementary symmetric function. In order
  to prove \eqref{eq:CutAndJoinGeneralRational}, it
  is enough to notice that
  \[ \sum_{i=0}^n e_{n-i}(\PP)\sum_{\gamma' \in \Gamma^{\geq 0}_{(0,-1) \to
                  (i+1,0)}}\widetilde{\wt}(\gamma'\,|0) = \sum_{\gamma \in \Gamma^{\geq 0}_{(0,-1) \to
          (n+1,0)}}\widetilde{\wt}(\gamma\,|(0,\PP)),\]
    and similarly
      \[ \sum_{i =0}^m e_{m-i}(-\QQ) \sum_{\gamma' \in \Gamma^{\geq 0}_{(0,-1) \to
                  (i+2,-1)}}\widetilde{\wt}(\gamma'\,|0) = \sum_{\gamma \in \Gamma^{\geq 0}_{(0,-1) \to
          (n+2,-1)}}\widetilde{\wt}(\gamma\,|(0,-\QQ,0).\]
    We will explain only the first equality, as the second one follows
    by the same argument. This identity follows directly by interpreting the product of the weights $\widetilde{\wt}(s_i\,|(0,\PP))$
             associated with the horizontal steps as the sum over all
             subsets of the horizontal steps, and then removing the selected
             steps at places $j_1 < \dots < j_{n-i}$ (thus,
             obtaining a path $\gamma' \in
             \Gamma^{\geq 0}_{(0,-1)\to(i+1,0)}$) and associating the weight
             $P _{j_1-1}\cdots
             P_{j_{n-i}-1} \cdot
             \widetilde{\wt}(\gamma'\,|0)$. Finally, the fact that
             there is a unique solution $\tau$ in the ring $\Q(\QQ)[\PP, \tilde{\pp}, \bb](\!( \hbar)\!)\llbracket
t\rrbracket$ which is slightly bigger than the ring $\Q[\tilde{\pp}, f_0^{-1},f_1,f_2,\dots,\hh,\bb](\!( \hbar)\!)\llbracket
t\rrbracket$ given in~\cref{theo:GeneralCutJoin} follows the 
argument given there verbatim. Indeed, using the same notation as in the proof
of~\cref{theo:GeneralCutJoin} it is enough to multiply $[t^k]\tau'$
by a polynomial $c(\QQ) \in \Q[\QQ]$ so that $[t^k]c(\QQ)\cdot\tau'
\in \Q[\QQ,\PP, \tilde{\pp}, \bb](\!( \hbar)\!)$. Then $c(\QQ)
\tau'$ also satisfies~\eqref{eq:CutAndJoinGeneralRational}. Then the
argument from
the proof
of~\cref{theo:GeneralCutJoin} can be applied with the grading $\deg(P_i) = \deg(Q_i)=1$.
           \end{proof}

\section{Airy structures from \texorpdfstring{$\mathcal W$}{W}-algebra representations}

The formalism of Airy structures provides a set of conditions under which a sequence of differential operators has a unique solution of a certain form. Typically, the differential operators arise from representations of vertex algebras, and the solution is a partition function that encodes interesting geometric information. In this section, we construct an Airy structure using the principal $ \mathcal W$ algebra of type $A $ whose  partition function will be related to weighted $b$-Hurwitz numbers.

\subsection{Airy structures}
We introduce the notion of Airy structures as studied in \cite{KontsevichSoibelman2018, BorotBouchardChidambaramCreutzigNoshchenko2024,BouchardCreutzigJoshi2024}. We shall mostly adopt the notation and presentation of Airy structures in \cite{BouchardCreutzigJoshi2024}, and refer the reader there for the proofs. 

Before defining Airy structures in general, we need to set up some notation and terminology. Consider a countable index set $A$. We will denote the set of variables $\{x_a\}_{a\in A} $ by $x_A$, the differential operator $\frac{\partial}{\partial x_a}$ by $\partial_a$ and the set of differential operators $\{ \frac{\partial}{\partial x_a}\}_{a\in A} $ by $\partial_A$.  We introduce a formal parameter $\hbar $ and the notion of $\hbar$-degree given by $\deg_\hbar  \hbar=1$, $\deg_\hbar x_a = 0$ and $\deg_\hbar \partial_a = 0$ for all $a\in A$, in order to distinguish from degree of polynomials in $x_A$ where we view $x_a$ as degree 1.

Using this $\hbar$-grading, we can define the \textit{completed Rees Weyl algebra} $\widehat{\mathcal{D}}_A^{\hbar} $ (see \cite[Definition 2.12]{BouchardCreutzigJoshi2024}), whose elements $ S \in \widehat{\mathcal{D}}_A^{\hbar}$ are differential operators of the form 	
\[
S = \sum_{n \in \mathbb Z_{\geq 0}} \hbar^n \sum_{\substack{m,k \in \mathbb Z_{\geq 0} \\ m+k = n} }\sum_{a_1,\ldots, a_m \in A} s^{(n,k)}_{a_1,\ldots, a_m}(x_A) \partial_{a_1}\ldots \partial_{a_m},
\] where the $ s^{(n,k)}_{a_1,\ldots, a_m}(x_A)  $ are polynomials of degree $ \leq k $.

\noindent We are ready to define an Airy structure in $\widehat{\mathcal{D}}_A^{\hbar}$ following \cite{BouchardCreutzigJoshi2024} now.
\begin{definition}\label{d:airy}
	Let $\mathcal{I} \subset \widehat{\mathcal{D}}_A^{\hbar}$ be a left ideal. We say that the ideal $ \mathcal I $  is an \emph{Airy structure} (or \emph{Airy ideal}) if there exist operators $\mathsf H_a \in \widehat{\mathcal{D}}_A^{\hbar}$, for all $a \in A$, such that:
	\begin{enumerate}[label = \alph*)]
		\item The collection of operators $\{ \mathsf H_a \}_{a \in A}$ is bounded (see \cite[Definition 2.15]{BouchardCreutzigJoshi2024} for a definition of boundedness\footnote{Boundedness allows us  to take infinite linear combinations of the operators $ H_a $ without divergent sums appearing, see \cite[Section 2.2]{BouchardCreutzigJoshi2024}.}).
		\item The left ideal $\mathcal{I}$ can be written as
		\begin{equation*}
			\mathcal{I} = \left \{ \sum_{a \in A} c_a  \mathsf H_a\ | \ c_a \in \widehat{\mathcal{D}}_A^{\hbar} \right \},
		\end{equation*}
		which consists of finite and infinite (if $A$ is countably infinite) $\widehat{\mathcal{D}}_A^{\hbar}$-linear combinations of the $\mathsf H_a$. 
		\item The operators $\mathsf H_a$ take the form
		\begin{equation*}\label{eq:form}
			\mathsf H_a = \hbar \partial_a  + O(\hbar^2).
		\end{equation*}
		\item The left ideal $\mathcal{I}$ satisfies the property:
		\begin{equation*}
			[ \mathcal{I}, \mathcal{I}] \subseteq \hbar^2 \mathcal{I}.
		\end{equation*}
	\end{enumerate}
\end{definition}

We need a slight generalization of the above notion of  Airy structures as we would like to incorporate so-called Whittaker shifts into our Airy structures (see \cref{sec:shiftedAirystructures}). More precisely, consider a set of formal commuting variables $\{\Lambda_a\}_{a\in A}$ with $\deg_\hbar\Lambda_a=0$. Then, we define a shifted Airy structure as follows.

\begin{definition}
	Let the collection of operators $\{\mathsf H_a\}_{a\in A}$ generate an Airy structure. Consider the ideal $\mathcal I^\Lambda \subset  \widehat{\mathcal{D}}_A^{\hbar}\llbracket\Lambda_A\rrbracket $ generated\footnote{In the context of Airy structures,  the word generated will always mean generated by infinite linear combinations.} by the operators $\mathsf H_a - \Lambda_a$ for every $a \in A$, viewed as elements of $\widehat{\mathcal{D}}_A^{\hbar}\llbracket\Lambda_A\rrbracket  $. Then the ideal $\mathcal I^\Lambda$ is said to be a \emph{shifted Airy structure} if it satisfies condition $(d)$ from \cref{d:airy}, i.e., 
	\[
		[ \mathcal{I}^\Lambda, \mathcal{I}^\Lambda] \subseteq \hbar^2 \mathcal{I}^\Lambda.
	\]
\end{definition}

The reason that (shifted) Airy structures are interesting is the following  theorem that guarantees a unique solution  $Z$ (of a certain form) to the equations $\mathsf H_a Z = \Lambda_a Z$ for all $a \in A $. 

\begin{theorem}\label{thm:pf}
	Let $\mathcal{I}^\Lambda \subset \widehat{\mathcal{D}}_{A}^{\hbar}\llbracket \Lambda_{A}\rrbracket $ be a shifted Airy structure. Then there exists a unique solution $ Z  $ of the form 
	\begin{equation*}\label{eq:pf}
		Z =\exp \left( \sum_{\substack{g \in \frac{1}{2} \mathbb Z_{\geq 0}, n \in \mathbb Z_{\geq 1}}} \hbar^{2g-2+n} F_{g,n}(x_A) \right)\,,
	\end{equation*} to the differential equations $\mathcal{I}^\Lambda \cdot Z = 0$ (\textit{i.e.}, $\mathsf H_a Z = \Lambda_a Z$, for all $a\in A$), given the initial condition $Z\big|_{x_A=0}=1$. The  $F_{g,n}(x_A)$ are formal power series in the $\Lambda_A$, whose coefficients are homogeneous polynomials of degree $n$ with $F_{g,n}(0) = 0$, i.e. $F_{g,n}(x_A) \in \mathbb C[x_A]\llbracket \Lambda_A\rrbracket  $. We call $Z$ the \emph{partition function} of the $\Lambda$-shifted Airy structure $\mathcal{I}^\Lambda$.
\end{theorem}

For Airy structures, this theorem was first proved in \cite{KontsevichSoibelman2018}. In the case of   shifted Airy structures, the generalized version presented above can be proved easily  following the arguments of \cite[Theorem 2.28]{BouchardCreutzigJoshi2024}, but a thorough treatment of shifted Airy structures will  appear in \cite{BorotBouchardChidambaramCreutzigNawata202x}. An alternative to working with shifted Airy structures is the approach of \cite{BorotBouchardChidambaramCreutzig2024} -- first,  rescale all the parameters $\Lambda_a$ by $\hbar^2$ to get an honest Airy structure and its associated partition function, and then undo the rescaling to recover the partition function of the shifted Airy structure.

Airy structures were invented as an algebraic framework in which to understand the Eynard--Orantin topological recursion and its variants. In general, it's not an easy task to construct examples of Airy structures. Almost all interesting examples of Airy structures, in which the partition function has a geometric interpretation, come from representations of vertex algebras such as $\mathcal W$-algebras  \cite{BorotBouchardChidambaramCreutzigNoshchenko2024, BorotBouchardChidambaramCreutzig2024}. As we will see in this paper, the Airy structures of interest for $b$-Hurwitz theory also come from $\mathcal W$-algebras.

\subsection{The algebra $\mathcal W^{\mathsf k}(\mathfrak{gl}_r)$}\label{sec:Walg} In this section, we give a brief introduction to the principal $\mathcal W$-algebra of $\mathfrak{gl}_r$ at level $\mathsf k$, which we will denote by $\mathcal W^{\mathsf k}(\mathfrak{gl}_r)$. A standard reference for $\mathcal W$-algebras and their representation theory is \cite{Arakawa2017}. In this paper, we will only study the algebra $\mathcal W^{\mathsf k}(\mathfrak{gl}_r)$ via its free field embedding into a Heisenberg vertex operator algebra (VOA).  Throughout this section, we work over the ring $\mathbb Q[\hbar]$, where $\hbar$ is a formal parameter that keeps track of the conformal dimension of the $\mathcal W$-algebra. 

A Heisenberg VOA of rank $r$, say $\mathcal H_r$, is strongly and freely generated as a vertex algebra by $r$ fields   $\mathsf J^1(z), \mathsf J^2(z),\ldots, \mathsf J^r(z)$, whose mode decomposition is 
\begin{equation*}
	\mathsf J^a(z) = \sum_{k\in \mathbb Z} \mathsf J^a_k z^{-k-1}, \qquad a = 1,\ldots, r\,,
\end{equation*}  and the modes satisfy the following commutation relations:
\begin{equation*}
	[\mathsf J^{a_1}_{k_1}, \mathsf J^{a_2}_{k_2}] = \hbar^2 k_1 \delta_{a_1,a_2}\delta_{k_1+k_2,0}.
\end{equation*}

At any level $\mathsf k \in \mathbb C$, the vertex algebra $\mathcal W^{\mathsf{k}}(\mathfrak{gl}_r)$ is also strongly and freely generated by $r$ fields $\mathsf W^1(z),\ldots, \mathsf W^r(z)$ of conformal dimensions $1,\ldots,r$ respectively. A convenient choice of these generators is provided by the quantum Miura transform, which embeds the algebra $\mathcal W^{\mathsf k}(\mathfrak{gl}_r)$ into the Heisenberg VOA $\mathcal H_r$. An explicit expression for the fields $\mathsf W^i(z)$ is given in \cite[Corollary 3.12]{ArakawaMolev2017}:
\begin{align}\label{eq:miura}
	\sum_{i=0}^r	\mathsf W^i(z) ( \hbar  \mathfrak{b})^{r-i} \partial_z^{r-i} =   \left( \hbar \mathfrak{b} \partial_z + \mathsf  J^1(z)\right)\cdots   \left(\hbar  \mathfrak{b} \partial_z + \mathsf J^r(z)\right),
\end{align} where we set $\mathsf W^0(z) = 1$ by convention. The parameter $\mathfrak{b}$ is related to the level $\mathsf k$ as $\mathfrak{b} = \mathsf k + r-1$. We use the following mode convention for the generating fields $\mathsf W^i(z) $, for $i = 1,\ldots, r$.
\begin{equation*}
	\mathsf W^i(z) = \sum_{k\in \mathbb Z} \mathsf W^i_k z^{-k-i}\,.
\end{equation*} For $ i = 1,2,3 $, \eqref{eq:miura}  yields the following formulae for the modes $\mathsf W^i_k$:
\begin{align*}
	\mathsf W^1_k &= \sum_{a_1=1}^r 	\mathsf  J^{a_1}_k ,\\
	\mathsf W^2_k &= \sum_{\substack{1\leq a_1 < a_2\leq r,\\
  k_1+k_2=k}} 	\mathsf  J^{a_1}_{k_1} 	\mathsf  J^{a_2}_{k_2}  - \bb\hbar(k+1) \sum_{a=1}^r  (a-1)\mathsf J^{a}_k  \\
	\mathsf W^3_k &= \sum_{\substack{1\leq a_1 < a_2 < a_3\leq r,\\
  k_1+k_2+k_3=k}}\mathsf J^{a_1}_{k_1} 	\mathsf J^{a_2}_{k_2} 	\mathsf J^{a_3}_{k_3}   - \bb\hbar \sum_{\substack{1\leq a_1 < a_2\leq r,\\
  k_1+k_2=k}} 	\big(k_1(a_1-1)+k_2(2a_1-a_2)+a_1+a_2-3\big)\mathsf J^{a_1}_{k_1}
  \mathsf  J^{a_2}_{k_2}+\\
  &+\bb^2\hbar^2(k+2)(k+1) \sum_{a=1}^r
    \frac{(a-1)(a-2)}{2}\mathsf  J^a_k.
\end{align*}
For higher values of $i $ and for generic $\bb$, explicit expressions for the modes $	\mathsf W^i_k$ that one obtains directly from \eqref{eq:miura} are rather complicated.

\subsubsection{A combinatorial description of the modes  $	\mathsf W^i_k$}
For our purposes, it's essential to find a tractable formula for these modes $	\mathsf W^i_k $, and this is the main result of this subsection. More precisely, we will provide a combinatorial formula for the modes $\mathsf W^i_k$ in terms of  weighted paths as introduced in \cref{sec:weightedpaths}.

In addition to the notion of paths considered already, we define colored paths, which are paths whose steps carry an associated integral weight.
\begin{definition}
	Fix an integer $r \geq 1$. A \emph{colored path} is a tuple $ (\gamma, f) $, where $\gamma$ is a path in $\Gamma_{(x,y) \to (x',y')}$, and  $f : [x'-x] \to [r] $ is a function that associates an integer  to each step of the path $\gamma$.
\end{definition}

For every colored path $(\gamma, f)$ with $\gamma$ in $\Gamma_{(x,y) \to (x',y')}$, we  associate a weight $\wt(\gamma,f)
:= \wt_1(\gamma,f)\cdots\wt_{x'-x}(\gamma,f)$ as
\[\wt_j(\gamma,f)
= \begin{cases} \J^{f(j)}_{\gamma_j} &\text{ if } \gamma_j
	\neq 0,\\
\J^{f(j)}_{0}-\mathfrak{b}\hbar\big(\sum_{l=j+1}^{x'-x}\gamma_l+x'-x-j\big)
	&\text{ if } \gamma_j= 0,\end{cases}\]
where we denote the increment of the $j$-th step of the path by
$-\gamma_j$, \textit{ i.e.,} if the $j$-th segment is from $(x_j,
y_j)$ to $(x_j+1,y'_j)$, then $\gamma_j = - (y'_j-y_j)  $. See
\cref{fig:Paths} for examples of paths and the associated weights,
including a combinatorial interpretation of the sum appearing in the
weight of a
horizontal step.

\begin{figure}[h]
	\centering
	\includegraphics[width=0.9\textwidth
	]{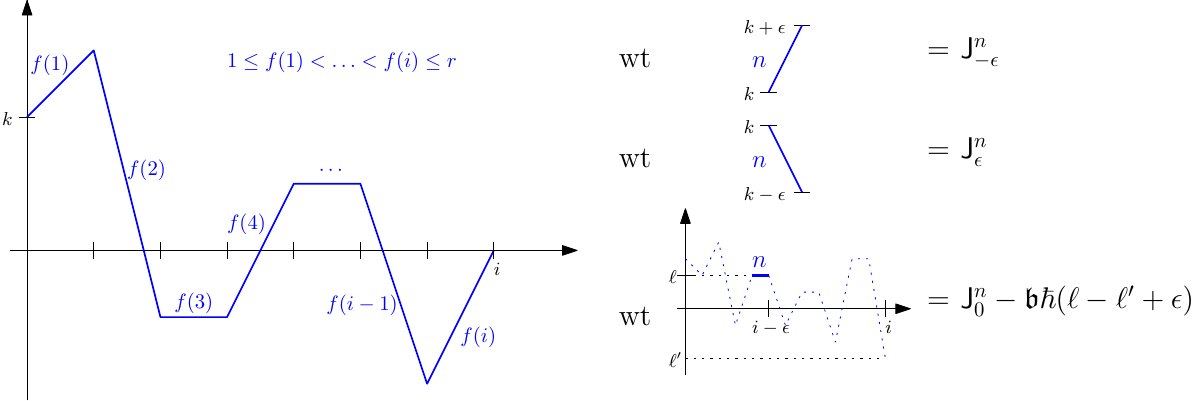}
	\caption{Generic colored path associated with the operators $\mathsf
          W^{i}_k$ by \eqref{eq:WkCombinatorial}. Here $k = 2, i=7$, and the associated weight is equal to $\J^{f(1)}_{-1}\J^{f(2)}_{4}(\J^{f(3)}_{0}-3\bb\hbar)\J^{f(4)}_{-2}(\J^{f(5)}_{0}-3\bb\hbar)\J^{f(6)}_{3}\J^{f(7)}_{-2}$.	}
	\label{fig:Paths}
\end{figure}

\noindent With this set up, we can prove the following combinatorial formula for the modes $ \mathsf W^{i}_k $.
\begin{lemma}\label{lem:WkCombinatorial}
	We have the following combinatorial interpretation of the modes $\mathsf W^{i}_k$ for all $1 \leq i \leq r$ and for all $k\in\mathbb Z$: 
	\begin{equation}
		\label{eq:WkCombinatorial}
		\mathsf W^{i}_k = \sum_{\substack{\gamma \in \Gamma_{(0,k) \to
				(i,0)},\\ f\in \FI([i],[r])}}\wt(\gamma,f),
                        \end{equation}
                        where we recall that $f\in \FI([i],[r])$
                      means that we are summing over strictly
                      increasing colorings.
\end{lemma}
 \begin{proof}
 	In this proof, we will need to work with the algebra $\mathcal{W}^{\mathsf k}(\mathfrak{gl}_{r-1})$ realized as an embedding into the Heisenberg VOA generated by the fields $\mathsf J^2(z),\ldots, \mathsf J^r(z)$. Then the algebra $\mathcal{W}^{\mathsf k}(\mathfrak{gl}_{r-1})$ is generated by $(r-1)$ generators, say $\mathsf W^{i,r-1}(z)$, defined by the following Miura transform:
	\begin{equation*}
		\sum_{i=0}^{r-1}\mathsf W^{i,r-1}(z)(\hbar\mathfrak{b}\partial_z)^{r-1-i}:= (\hbar\mathfrak{b}\partial_z+\J^2(z))\cdots(\hbar\mathfrak{b}\partial_z+\J^r(z))\,.
	\end{equation*} In this proof, we denote the generators of $\mathcal{W}^{\mathsf k}(\mathfrak{gl}_{r})$ given by the Miura transform  \eqref{eq:miura} as $ \mathsf W^{i,r}(z)$, to emphasize the $r$-dependence and avoid confusion. Then, the quantum Miura transform for  $\mathcal{W}^{\mathsf k}(\mathfrak{gl}_{r})$ can be written as 
	\begin{equation*}
		\sum_{i=0}^{r}{\mathsf W}^{i,r}(z)(\hbar\mathfrak{b}\partial_z)^{r-i}= \left(\hbar\mathfrak{b}\partial_z+\J^1(z)\right)\cdot \sum_{i=0}^{r-1}\mathsf W^{i,r-1}(z)(\hbar\mathfrak{b}\partial_z)^{r-1-i},
	\end{equation*}
	which, in terms of the modes, gives
	\begin{equation}
		\label{eq:WOpRec}
		{\mathsf W}^{i,r}_k=\mathsf W^{i,r-1}_k+\sum_{k_1+k_2=k} \left(\J^1_{k_1}-\delta_{k_1,0}\hbar\mathfrak{b}(k_2+i-1)\right)\cdot\mathsf{W}^{i-1,r-1}_{k_2}.
	\end{equation}
	In addition to our existing convention that $\mathsf W^{0,r}_k = \delta_{k,0} $, we  set $\mathsf W^{r,r-1}_k=0$ and $\mathsf{W}^{-1,r-1}_k=0$ for any $k \in \mathbb Z$.

	With this setup, we can prove \eqref{eq:WkCombinatorial} by induction with respect to the lexicographic order on $(i,r)$. It is clear that the formula holds for $(1,r)$ for all $r$. Assume that it is true for all $(i',r') < (i,r)$. Then, the RHS of \eqref{eq:WkCombinatorial} can be decomposed into two terms -- either the first step is colored by $f(1)=1$, or $f(1)>1$. When $f(1)=1 $, the weight $\wt_1(\gamma,f) =
	\J^{1}_{k_1}-\delta_{k_1,0}\mathfrak{b}\hbar\big(k_2+i-1\big)$,
	where  the first step of $\gamma$ ends at $(1,k_2)$ and $k_1 = k - k_2$. This gives the following formula for the RHS of \eqref{eq:WkCombinatorial}:
	\begin{equation}
		\label{eq:CombRecW}
		\sum_{\substack{\gamma \in \Gamma_{(0,k) \to
				(i,0)},\\ f\in \FI([i],[2..r])}}\wt(\gamma,f) +\sum_{k_1+k_2=k}\left(\J^1_{k_1}-\delta_{k_2,0}\hbar\mathfrak{b}(k_2+i-1)\right)\cdot \sum_{\substack{\tilde{\gamma} \in \Gamma_{(0,k_2) \to
				(i-1,0)},\\ \tilde{f}\in \FI([i-1],[2..r])}}\wt(\tilde{\gamma},\tilde{f}),
	\end{equation}
	where $\tilde{\gamma}$ is the path obtained from $\gamma$ by removing the first step and shifting it one unit to the left, so that $\tilde{\gamma}$ starts at $(0,k_2)$. Note that the weight of a path is
	invariant under a horizontal shift.  For $j >1$, we define  $\tilde{f}(j-1) := f(j)$ such that $\wt_j(\gamma,f) = \wt_{j-1}(\tilde{\gamma},\tilde{f})$. The inductive hypothesis gives us the following formula for
	$j \leq i$: 
	\begin{equation*}
		\label{eq:WShiftedComb}
		\mathsf W^{j,r-1}_k = \sum_{\substack{\gamma \in \Gamma_{(0,k) \to
				(j,0)},\\ f\in \FI([j],[2..r])}}\wt(\gamma,f),
	\end{equation*}
	and hence \eqref{eq:CombRecW} is equivalent to the recursion \eqref{eq:WOpRec}, which finishes the proof.
\end{proof}

\subsubsection{Algebra of modes} In order to construct an Airy structure, we will need to work with the algebra of modes of the vertex algebra $\mathcal{W}^{\mathsf k}(\mathfrak{gl}_r)$. We consider the algebra $\mathcal A^\hbar$, which is a suitable completion of the algebra of modes, with the parameter $\hbar$ keeping track of the conformal dimension of the generating fields (see \cite[Section 3.1]{BorotBouchardChidambaramCreutzig2024} for a definition). 

The strategy to find an Airy structure is to pick a set of modes of $\mathcal{W}^{\mathsf k}(\mathfrak{gl}_r)$, such that the ideal generated by these modes will be a subalgebra of $\mathcal A^\hbar$.  In addition, we may want to pick a  subset of these modes, and shift them to obtain Whittaker vectors. Various such subsets of modes were studied in \cite{BorotBouchardChidambaramCreutzigNoshchenko2024}, and we present one of these subsets that is relevant for us. In fact, we prove an upgraded version that allows us to shift the zero modes $\mathsf W^i_0$ later on.

\begin{lemma}\label{lem:subalg}
	Consider the set of modes $S$ defined as 
	\[
		S := 	\left\{ \mathsf W^i_k : i \in [r]\,, k \geq 0 \right\},
		\]  and  the subset of modes $S' := 	\left\{ \mathsf W^i_k : i \in [r]\,, k \geq 1 \right\} \subset S$. Then, for any two modes  $\mathsf W^{i_1}_{k_1} , \mathsf W^{i_2}_{k_2} $ in $S$, we have 
		\[
			[\mathsf W^{i_1}_{k_1} , \mathsf W^{i_2}_{k_2}] \in \hbar^2 \mathcal A^\hbar \cdot S'.
		\]
\end{lemma}
\begin{proof}
	The weaker statement that  $[\mathsf W^{i_1}_{k_1} , \mathsf W^{i_2}_{k_2}] \in \hbar^2 \mathcal A^\hbar \cdot S $ is  a direct application of \cite[Proposition 3.14]{BorotBouchardChidambaramCreutzigNoshchenko2024} to the strong generators of $\mathcal{W}^{\mathsf k}(\mathfrak{gl}_r)$ (up to the rescaling by $\hbar$). 
	
	To prove the lemma, we  need to show that terms appearing in the commutator 	$[\mathsf W^{i_1}_{k_1} , \mathsf W^{i_2}_{k_2}]$ can be ordered such that the rightmost mode of each monomial is not $\mathsf W^i_0$. Using Li's filtration \cite{Li2005}, any monomial can be normally ordered to take the form (up to a constant prefactor)
	\[
		\mathsf W^{j_1}_{\ell_1}\cdots \mathsf W^{j_n}_{\ell_n},
	\] for some $n \geq 1$, where $\ell_n = \max(\ell_1,\dots, \ell_n)$ and $\ell_1+\cdots+\ell_n = k_1+k_2$.
	 If $k_1 + k_2 \geq 1$, we must have $\ell_n \geq 1$, and thus the zero modes do not appear as the rightmost mode in the monomial. 
	 
	 For the case $k_1=k_2 = 0$, we need to use the fact that the Zhu algebra of $\mathcal W^{\mathsf k}(\mathfrak{gl}_r)$ is commutative \cite[Theorem 4.16.3]{Arakawa2017}. Recall that the Zhu algebra in general is defined as a certain quotient of the algebra of zero modes \cite[Section 3.12]{Arakawa2017}. In the case of $\mathcal W^{\mathsf k}(\mathfrak{gl}_r)$, the Zhu algebra is generated by monomials of the form $ 	\mathsf W^{i_1}_{k_1}\cdots \mathsf W^{i_n}_{k_n} $ such that $\sum_{i=1}^n k_i = 0$, modulo relations $ \mathsf W^{j_1}_{\ell_1}\cdots \mathsf W^{j_n}_{\ell_n} \cdot	\mathsf W^{j_{n+1}}_{\ell_{n+1}}\cdots \mathsf W^{j_{n+m}}_{\ell_{n+m}} = 0$ whenever $\sum_{i=1}^{n} \ell_i  = - \sum_{j=n+1}^{n+m} \ell_j <0 $.
	 
	 Then, we know that the commutator $[\mathsf W^{i_1}_0, \mathsf W^{i_2}_0] $ is a linear combination of terms of the form $ 	\mathsf W^{j_1}_{\ell_1}\cdots \mathsf W^{j_n}_{\ell_n} \cdot	\mathsf W^{j_{n+m}}_{\ell_{n+1}}\cdots \mathsf W^{j_{n+m}}_{\ell_{n+m}}$, where $\sum_{i=1}^{n} \ell_i  = -\sum_{j=n+1}^{n+m} \ell_j  < 0 $. Again using Li's filtration, we can order $\mathsf W^{j_{n+1}}_{\ell_{n+1}}\cdots \mathsf W^{j_{n+1}}_{\ell_{n+m}}$ such that $\ell_{n+m} >0$, and hence we are done. 
	\end{proof}

\subsection{An Airy structure for weighted $b$-Hurwitz numbers}Now that we have introduced the notion of Airy structures and the algebra $\mathcal W^{\mathsf k}(\mathfrak{gl}_r)$, we can go ahead and construct the Airy structures we are after.  We will consider a specific representation of the vertex algebra $\mathcal W^{\mathsf k}(\mathfrak{gl}_r)$, look at a certain subset of modes and prove that they form an Airy structure. This Airy structure will turn out to have a trivial partition function i.e. $Z=1$, but by doing appropriate shifts, we will obtain a shifted Airy structure with an interesting partition function.

\begin{remark}\label{rem:weight}
	It may help the reader to keep in mind that the partition function of the shifted Airy structure that we will eventually construct will be related to the generating function of weighted $b$-Hurwitz numbers with the following weight:
	\[
	G(z) = \frac{\prod_{i=1}^r (P_i+ z ) }{\prod_{i=1}^{r-2}(Q_i-z )}\,
	\] where $r \geq 2$ is an integer. While this is not the most general weight we are interested in, one can take appropriate limits to reduce to an arbitrary rational weight (see \cref{sec:Dlimits} for a more detailed explanation).
\end{remark}

We will first construct a representation of the Heisenberg  VOA $\mathcal{H}_r$. Consider the ring $ R = \mathbb C(\mathbf P, \mathbf Q, \mathfrak{b})[\mathbf x] $, which is the polynomial ring in infinitely many variables $ \mathbf x =  \{x^a_k\}_{k\geq 1,  a \in[r]} $ over the field of rational functions in $\mathfrak{b}$, $\mathbf P = \{P_1,\dots, P_r\}$ and $\mathbf Q = \{Q_0,Q_1,\dots, Q_{r-2}\}$ --- we will set $Q_0=0$ in Section \ref{sec:bHurwitzfromZ}. 
Consider the  vector space 
\[
\widehat{\mathcal M} := \prod_{n\in \mathbb Z_{\geq 0}}\hbar^n \left(R_{\leq n}  \right),
\] where $R_{\leq n} $  denotes polynomials in the $ x^a_k$ of degree $\leq n$. A typical element of $\widehat{\mathcal M} $ is of the form $\sum_{n =0}^\infty \hbar^n P_n$, where $P_n$ is a polynomial in the $ x^a_k $ of degree at most $n$. This vector space $ \widehat{\mathcal M} $ is a module for the Heisenberg VOA via the following mode assignment:

\begin{equation}
	\label{eq:Jrep}
	\begin{split}
		\mathsf J^a_{k} = \left\{
		\begin{array}{lr}
			\hbar  \partial_{x^a_k}& k > 0,\\
			\hbar (-k)x^{a}_{-k} +(-1)^{r}\frac{1}{\Lambda} \delta_{k,-1} \delta_{a,r} & k < 0, \\
			 Q_{a-1}-\hbar \mathfrak{b}(a-1)  & a \in [r-1] ,k = 0\, \\
			 -\sum_{a'=1}^r P_{a'} -\sum_{a'=0}^{r-2} Q_{a'} -\hbar \mathfrak{b}(a-1) & a=r ,k = 0, \\
		\end{array}
		\right.
	\end{split}
\end{equation} Note that while the $\J^1_k$ have been defined previously in \eqref{eq:J1def}, it is not a conflict of notation as we will eventually substitute $\tilde{p}_k = k x^1_k$ (see \cref{thm:tauZidentification}). As  the quantum Miura transformation gives an embedding $\mathcal W^{\mathsf k}(\mathfrak{gl}_r) \hookrightarrow \mathcal H_r$, $\widehat{\mathcal M}$ is also a module for   $\mathcal W^{\mathsf k}(\mathfrak{gl}_r) $ by restriction. By abuse of notation, we will henceforth use both $\mathsf W^i_k$ and $ \J^a_k$ to denote the modes of the generators of $\mathcal W^{\mathsf k}(\mathfrak{gl}_r) $ and $\mathcal H_r$, respectively, in the representation $\widehat{\mathcal M}$.

We would like to claim that the set of operators $\left\{ \mathsf W^i_k : i \in [r], k \geq \delta_{i,1} \right\}  $ in the representation $ \widehat{\mathcal M}$ generates an Airy
structure. However, this set of operators does not satisfy
condition $(c)$ of \cref{d:airy} of Airy structures. Indeed, for an
operator $P$ in $\widehat{\mathcal{D}}_A^{\hbar}$, let $\pi_i(P) :=
\hbar^i\cdot[\hbar^i]P$ denote the $\hbar$-degree $i$ term of $P$. Then we have the following two problems: first $\pi_0(\mathsf W^i_0) \neq 0$, and second $ \pi_1(\mathsf W^i_k)
\neq \hbar \partial_{x^{\bullet}_{\bullet}}$. 

To address the first issue, let us define new operators $\widetilde{\mathsf W}^i_k$, where we remove some terms by hand, including all the $\hbar$-degree zero terms. Define $\mathsf V^i$ for any $i \in [r]$ as the following combinations of the Heisenberg zero modes:
\begin{equation*}\label{eq:defVi0}
	\mathsf V^i := \sum_{f \in \FI([i],[r])}\big(\J^{f(1)}_0-\hbar\bb(i-1) \big)\cdots \big (\J^{f(i-1)}_0-\hbar\bb \big) \J^{f(i)}_0.
      \end{equation*}
      Define, for any $i \in [r]$, and any $k \geq 0$,
\begin{equation}\label{eq:Wtilde}
	\widetilde{\mathsf W}^i_k := \mathsf W^i_k - \mathsf V^i \delta_{k,0}.
\end{equation}
	To address the second issue, we need to diagonalize the set of
        operators $\left\{ \widetilde{\mathsf W}^i_k : i \in [r], k \geq \delta_{i,1} \right\}  $. Assuming that  $ Q_0, Q_1,\dots, Q_{r-2}$
      are pairwise distinct, define  operators $\widetilde{\mathsf H}^a_{k}$, for any $ a \in [r]$ and any $k \geq 1$:
      \begin{align}
        \widetilde{\mathsf H}^r_k := \widetilde{\mathsf W}^1_{k} +  \sum_{a \neq r} \sum_{\ell=0}^{k-1}\frac{((-1)^{r}\Lambda)^{\ell+1}\left(-e_1(\PP) -e_1(\QQ)-Q_{a-1}\right)^{\ell}}{  \prod_{a'\neq a,r}\left(Q_{a'-1}-Q_{a-1}\right)} \sum_{i=1}^r (-Q_{a-1})^{r-i} \widetilde{\mathsf W}^i_{k-1-\ell}\,&.\label{eq:defHak}\\
\widetilde{\mathsf H}^a_{k} := - \sum_{\ell=0}^{k-1}  \frac{((-1)^{r}\Lambda)^{\ell+1}\left(-e_1(\PP) -e_1(\QQ)-Q_{a-1}\right)^{\ell}}{  \prod_{a'\neq a,r}\left(Q_{a'-1}-Q_{a-1}\right)} \sum_{i=1}^r (-Q_{a-1})^{r-i} \widetilde{\mathsf W}^i_{k-1-\ell}, \ \text{ for } a \in[r-1]&. \nonumber	
\end{align} 

These operators $\widetilde{\mathsf H}^a_{k}$ are diagonal in the following sense.

\begin{lemma}\label{lem:deg1}
For any $1 \leq a \leq r$ and $k\geq 1$, we have 
\[
	\widetilde{\mathsf H}^a_{k} = \hbar \partial_{x^a_k}+O(\hbar^2)\,.
\] 
\end{lemma}
\begin{proof}
  First notice that $\pi_0(\J^a_k) = 0$ for all $k> 0$. \cref{lem:WkCombinatorial} implies that $\mathsf
  W^i_k$ can be written as a linear combination of operators
  associated with colored paths and the only paths in
  \eqref{eq:WkCombinatorial} that contribute in  $\hbar$-degree zero are paths
  that do not have any steps of negative increment. Since $k \geq 0$,
  the only such path appears when $k=0$ and it is the path that runs horizontally from $(0,0)$ to $(i,0)$. Therefore
  \eqref{eq:WkCombinatorial} implies that
  \[ \mathsf
  W^i_k = \delta_{k,0}\sum_{f \in \FI([i],[r])}\big(\J^{f(1)}_0-\hbar\bb(i-1) \big)\cdots \big (\J^{f(i-1)}_0-\hbar\bb \big) \J^{f(i)}_0+ O(\hbar) = \delta_{k,0}\mathsf V^i +
  O(\hbar), \]
which further implies that $\pi_0(\widetilde{\mathsf W}^i_k) = 0$. 
Similarly, the only paths in
  \eqref{eq:WkCombinatorial} that contribute in $\hbar$-degree one are paths
  that have precisely one step of negative increment. Moreover,
  $\pi_0(\J^a_k) = 0$ for all $k < 0$ except for $a=r$ and $k=-1$. Therefore the only paths in
  \eqref{eq:WkCombinatorial} that contribute in $\hbar$-degree one  are paths
  that have precisely one step of negative increment -- either they
  finish with a step from $(i-1,-1)$ to $(i,0)$ colored by $r$, or they have
  precisely one step of negative increment and all the other steps are flat. Thus, \eqref{eq:WkCombinatorial} implies that the
  corresponding $\hbar$-degree one part is given by 
\[\pi_1(\widetilde{\mathsf W}^i_k) =
                        \frac{(-1)^{r} }{\Lambda} \sum_{a=1}^{r-1}
                        \hbar\partial_{x^a_{k+1}}
                        \pi_0(e_{i-2}(\J^1_0,\dots,
                        \hat{\J^a_0},\dots,\J^{r-1}_0)) + \sum_{a=1}^r
                        \hbar\partial_{x^a_{k}}
                        \pi_0(e_{i-1}(\J^1_0,\dots,
                        \hat{\J^a_0},\dots,\J^r_0))\delta_{k>0},\]
                      for any $ i \in [r] $ and $k \geq 0$.
	In generating series form, we can express this formula as 
	\begin{equation*}
		\sum_{i=1}^r u^{r-i}	\pi_1(\widetilde{\mathsf W}^i_k) = \frac{(-1)^{r} }{\Lambda}\sum_{a=1}^{r-1}\hbar\partial_{x^a_{k+1}}\prod_{a'\neq a,r}(u + \pi_0(\J^{a'}_0)) + \sum_{a=1}^r\hbar\partial_{x^a_{k}}\prod_{a'\neq a} (u + \pi_0( \J^{a'}_0)) \delta_{k>0} \,.
	\end{equation*} For any $a \neq r$ we can invert it by dividing by $ \prod_{a'=1}^r (u + \pi_0( \J^{a'}_0))$ and
        taking the residue at $u = -\pi_0(\J^a_0)$, which for $k\geq 0 $ gives 
	\begin{equation} \label{eq:36}
		\frac{-(-1)^r\Lambda\pi_0(\J^r_0-\J^a_0)}{\prod_{a'\neq a} \pi_0(\J^{a'}_0-\J^a_0)} \sum_{i=1}^r (-\pi_0(\J^a_0))^{r-i} \pi_1(\widetilde{\mathsf W}^i_{k}) = \hbar\partial_{x^a_{k+1}} + (-1)^r \Lambda \pi_0(\J^r_0-\J^a_0) \hbar\partial_{x^a_{k}} \delta_{k > 0}.
              \end{equation}
              We complete diagonalization by removing the
              term $(-1)^r\Lambda \hbar\partial_{x^a_{k}}
              \pi_0(\J^r_0-\J^a_0)\delta_{k > 0}$ on the RHS of
              \eqref{eq:36}, which leads to the definition of
              $\widetilde{\mathsf H}^a_k$ for $a \neq r$ by comparing
              with \eqref{eq:Jrep}. When $ a =
              r$, and $k \geq 1$ we have $\J^r_k = \widetilde{\mathsf W}^1_k - \sum_{a=1}^{r-1}
              \J^{a'}_k$ for any $k \geq 1$, hence we define
              $\widetilde{\mathsf H}^r_k = \widetilde{\mathsf W}^1_k - \sum_{a'=1}^{r-1} \widetilde{\mathsf H}^{a'}_k$.
\end{proof}
Finally, we are ready to construct an Airy structure.
\begin{proposition}\label{prop:AS}
	Assume that $ Q_0, Q_1,\dots, Q_{r-2}$ are pairwise distinct. Then, the ideal $\mathcal I$ generated by the set of operators 
	\begin{equation*}\label{eq:ASoperators}
	\widetilde{	\mathsf  W}^i_k \qquad \text { for }  i \in [r] \,,\, k\geq \delta_{i,1}
	\end{equation*} is an Airy structure.
\end{proposition}
\begin{proof}  We will work with the diagonalized set of operators $\widetilde{\mathsf H}^a_k$ as defined in \eqref{eq:defHak}. 		The set of  operators 
	\begin{equation}\label{eq:setH}
	\left\{\widetilde{\mathsf H}^a_k : a \in [r]\,,k\geq 1\right\}
	\end{equation} 
	 also generate the ideal $\mathcal I$ as the change of basis  \eqref{eq:defHak} is invertible.

	We need to  check the conditions $(a)-(d)$ to be an Airy structure from \cref{d:airy}.
	\begin{enumerate}[label=(\alph*)]
		\item The operators $ \widetilde{\mathsf W}^i_k $ are bounded by \cite[Lemma 2.43]{BouchardCreutzigJoshi2024}, which proves that any set of non-negative modes of a vertex algebra obtained  from a representation of a Heisenberg VOA is bounded. The operators $\widetilde{\mathsf H}^a_k $ are also bounded as they are finite linear combinations of the $ \widetilde{\mathsf W}^i_k $.
		\item As we mentioned at the start of the proof, the ideal $\mathcal I$ is generated by the set of operators \eqref{eq:setH}.
		\item $\widetilde{\mathsf H}^a_{k} = \hbar
                  \partial_{x^a_k}+O(\hbar^2)$ by \cref{lem:deg1}.
		\item \cref{lem:subalg} shows that the ideal generated by the set of operators $\left\{ \mathsf W^i_k   : i \in [r], k \geq 0	\right\} $ satisfies  condition $(d)$ to be an Airy structure. In fact, it also allows us to shift the subset of operators $\{\mathsf W^i_0: i\in [r]\}$ to get the subset $\{\widetilde{\mathsf W}^i_0: i\in [r]\}$, and the ideal $\mathcal I$ generated by the resulting operators will still satisfy condition $(d)$. (Note that $\widetilde{\mathsf W}^1_0 = 0$, and hence we may disregard it.) \qedhere
	\end{enumerate}
	\end{proof}
	
	While we have constructed an Airy structure, the partition function associated to it can easily seen to be trivial. As we know that there is a unique solution by \cref{thm:pf} all that needs to be done is to check that $Z = 1$ is a solution to the set of equations
	\[
		\mathsf W^i_k Z = \mathsf V^i \delta_{k,0} Z \qquad
                \text { for }  i \in [r] \,,\, k\geq \delta_{i,1}\,.
	\] As this is a straightforward check, and we do not need this result in this paper, we leave it as an exercise for the interested reader.
	
	\subsubsection{Shifted Airy structure}\label{sec:shiftedAirystructures}
	
	In order to get the partition function that we are interested in, we need to shift the zero modes $\mathsf{W}^i_0$. More precisely, we have the following theorem.
	\begin{theorem}\label{thm:shiftedAS}
		Assume that $ Q_0, Q_1,\dots, Q_{r-2}$ are pairwise distinct. Then, the ideal $\mathcal I^L$ generated by the set of operators 
		\begin{equation}\label{eq:WopAS}
		\mathsf  W^i_k - \mathsf V^i \delta_{k,0} - L_i \delta_{k,0} \qquad \text { for }  i \in [r] \,,\, k\geq \delta_{i,1}\,,
		\end{equation} where $\left(L_i\right)_{i\in[2..r]}$ are formal commuting variables of $\deg_\hbar L_i=0$, is a shifted Airy structure.
	\end{theorem}
	\begin{proof}
		We know from \cref{prop:AS}, that the  ideal generated by the operators $	\mathsf  W^i_k - \mathsf V^i \delta_{k,0}$ is an Airy structure. So all we need to check to get a shifted Airy structure is that the condition $	[ \mathcal{I}^L, \mathcal{I}^L] \subseteq \hbar^2 \mathcal{I}^L $ is  satisfied.  \cref{lem:subalg} does the job, as it says that the modes $\mathsf W^i_0$ can be shifted freely without changing the commutation relations.
	\end{proof}
	Recall that by \cref{thm:pf} of Airy structures, the  shifted Airy structure $\mathcal I^L$ has a unique partition function, say $\mathcal Z^L$ of the following form:
	\[
		\mathcal Z^L =\exp \left( \sum_{\substack{g \in \frac{1}{2} \mathbb Z_{\geq 0}, n \in \mathbb Z_{> 0}}} \hbar^{2g-2+n} F_{g,n} \right)\,
	\] where the $F_{g,n}$, \emph{a priori}, are elements of $\mathbb C(\!( \mathbf P, \mathbf Q,  \Lambda)\!)(\mathfrak b)[\mathbf x]\llbracket L_2,\ldots, L_r\rrbracket $, and in this case the partition function $\mathcal Z^L$ will indeed be non-trivial.

	We  need to study the properties of the partition function $\mathcal Z^L $, especially the dependence of the $F_{g,n}$ on the  $L_i$, as we will want to specialize the $L_i$ to functions of $\mathbf P $ and $\mathfrak{b}$ later on (see \eqref{eq:Lambdasub}). Let us define the expansion coefficients of $F_{g,n}$ as follows:
	\begin{equation*}\label{eq:Fgnexpansion}
		F_{g,n} =  \frac{1}{n!} \sum_{\substack{a_1,\ldots, a_n\in [r] \\ k_1,\ldots, k_n \in \mathbb Z_{>0}}}  F_{g,n}\big[\begin{smallmatrix}a_1 & \cdots & a_n \\ k_1 & \cdots & k_n \end{smallmatrix}\big] \prod_{i=1}^n x^{a_i}_{k_i}.
	\end{equation*} such that  $F_{g,n}\big[\begin{smallmatrix}a_1 & \cdots & a_n \\ k_1 & \cdots & k_n \end{smallmatrix}\big]$ is an element of $\mathbb C(\!(\mathbf P,\mathbf Q,  \Lambda)\!)(\mathfrak b)\llbracket L_2,\ldots, L_r\rrbracket $. We can  show that the $F_{g,n}\big[\begin{smallmatrix}a_1 & \cdots & a_n \\ k_1 & \cdots & k_n \end{smallmatrix}\big]$  actually live in $\mathbb Q( \mathbf Q)[\mathbf P,\mathfrak{b},\Lambda][L_2,\ldots, L_r] $:

	\begin{proposition}\label{prop:Fgnbehaviour}
		The $F_{g,n}\big[\begin{smallmatrix}a_1 & \cdots & a_n \\ k_1 & \cdots & k_n \end{smallmatrix}\big]$ of the partition function  depend polynomially on $\Lambda, \mathfrak b,  \mathbf P$ and $L_2,\cdots, L_r$ and rationally on $\mathbf{Q}$. More precisely, we have 
			\[
			\frac{F_{g,n}\big[\begin{smallmatrix}a_1 & \cdots & a_n \\ k_1 & \cdots & k_n \end{smallmatrix}\big]}{\Lambda^{k_1 + \cdots + k_n}} \in \mathbb Q( \mathbf Q)[\mathbf P,\mathfrak{b}][L_2,\ldots, L_r] .
			\] 
	\end{proposition}

	\begin{proof}
In this proof we will
denote by $(I|A)$ a tuple of two ordered sets $I  = \{i_1,\dots,i_{d}\}\subset
\Z^{d}_{> 0}$, and $A = \{a_1,\dots,a_{d}\}\subset
[r]^{d}$ of the same cardinality $|(I|A)| := d$. For
such a tuple we denote $\xx^A_I := \prod_{j=1}^{|A|}x_{i_j}^{a_j}$,
and $\partial_{\xx^A_I}:=
\prod_{j=1}^{|A|}\partial_{x_{i_j}^{a_j}}$. Equations \eqref{eq:WkCombinatorial}, \eqref{eq:Jrep}, and \eqref{eq:Wtilde}
imply that
\[
  \Lambda\widetilde{\mathsf{W}}_{k}^{i}=\sum_{d+h=1}^i\sum_{|(I_1|A_1)\sqcup
    (I_2|A_2)|=d} \hbar^hK_{k,h}^{i}\big[\begin{smallmatrix}{A}_1
    &{A}_2 \\ {I}_1&{I}_2 \end{smallmatrix}\big] \cdot \hbar^d \cdot\xx^{A_1}_{I_1}\cdot \partial_{\xx^{A_2}_{I_2}},\]
where $K_{k,h}^{i}\big[\begin{smallmatrix}{A}_1
    &{A}_2 \\ {I}_1&{I}_2 \end{smallmatrix}\big]
  \in\mathbb{Q}[\PP,\QQ,\bb, \Lambda]$ vanish unless $\sum_{i'' \in
    {I}_2}i''-\sum_{i'\in {I}_1}i' \leq k+1$.
Therefore \eqref{eq:defHak} and \cref{lem:deg1} further imply that
\begin{equation}
  \label{eq:Hexpansion}
  \hbar\partial_{x_{k}^{a}}=\widetilde{\mathsf H}_{k}^{a}+\sum_{d+h=2}^r\sum_{|(I_1|A_1)\sqcup
    (I_2|A_2)|=d} \hbar^hU_{k,h}^{a}\big[\begin{smallmatrix}{A}_1
    &{A}_2 \\ {I}_1&{I}_2 \end{smallmatrix}\big] \cdot \hbar^d
  \cdot\xx^{A_1}_{I_1}\cdot \partial_{\xx^{A_2}_{I_2}},
  \end{equation}
where $U_{k,h}^{a}\big[\begin{smallmatrix}{A}_1 &{A}_2 \\
  {I}_1&{I}_2 \end{smallmatrix}\big]
\in\mathbb{Q}(\mathbf{Q})[\mathbf{P},\mathfrak b,\Lambda]$ vanish
unless $\sum_{i''\in {I}_2}i''-\sum_{i'\in {I}_1}i'\leq
k$. Moreover \eqref{eq:defHak} together with the shifted Airy
structure~\eqref{eq:WopAS} gives that
\begin{equation}
  \label{eq:HAiry}
  \frac{1}{\mathcal{Z}^L}\widetilde{\mathsf H}_{k}^{a}\mathcal{Z}^L =\widetilde{U}_k^a \in  \Lambda^k\cdot \Q(\QQ)[\PP,L_2,\dots,L_r]
\end{equation}
for any $a \in [r], k \geq 1$, where the coefficients
$\widetilde{U}_k^a$ can be written
explicitly from \eqref{eq:defHak}.

Note that we can extract the  coefficients of $\mathcal Z^{L}$ as follows
\[ F_{g,n}\big[\begin{smallmatrix}a_1 & \cdots & a_n \\ k_1 & \cdots &
  k_n \end{smallmatrix}\big] =
[\hbar^{2g-1+n}]\left(\prod_{l=2}^n\partial_{x_{k_l}^{a_l}}\frac{1}{\mathcal{Z}^L}(\hbar\partial_{x_{k_1}^{a_1}})\mathcal{Z}^L\right)\Bigg|_{\mathbf{x}=0}.\]
Applying \eqref{eq:Hexpansion} with $k=k_1,
a=a_1$, and comparing it with \eqref{eq:HAiry}, we obtain the
following identity
\begin{align}
F_{g,n}\big[\begin{smallmatrix}a_1 & \cdots & a_n \\ k_1 &
  \cdots &
  k_n \end{smallmatrix}\big]=\delta_{n,1}\delta_{g,0}\widetilde{U}_{k_1}^{a_1}+&\sum_{d+h=2}^r\sum_{l=1}^d\sum_{\substack{|(I|A)\sqcup
                                                                                 (I_1|A_1)
                                                                                 \sqcup\cdots
                                                                                 \sqcup(I_l|A_l)|=d,\\|(I_1|A_1)|,\cdots,|(I_l|A_l)|\geq
  1}}\cdot \nonumber \\
 \sum_{\substack{(I'_1 |A'_1)\sqcup\cdots\sqcup (I'_l
  |A'_l)\sqcup ({I}
  |{A})=\\=(\{k_2,\dots,k_n\}|\{a_2,\dots,a_n\})}}&\sum_{\substack{g_1,\cdots,g_l\in
                                                    \frac{1}{2}\Z_{>
                                                    0},\\2\sum_{m}
  (g_m-1+|{I}_m|)=2g-h}}
  U_{k_1,h}^{a_1}\big[\begin{smallmatrix}{A} &{A}_1\cdots{A}_l
    \\ {I}&{I}_1 \cdots{I}_l\end{smallmatrix}\big] \prod_{m=1}^l F_{g_m,|I'_m|+|{I}_m|}\big[\begin{smallmatrix}A'_m  & {A}_m \\ I'_m & {I}_m \end{smallmatrix}\big].\label{eq:Frecursion}
\end{align}

While the formula above looks complicated (it follows by applying the Leibnitz rule repeatedly and we leave its verification to
the interested reader), we only need the following
property: the RHS is a finite combination over $\mathbb Q( \mathbf Q)[\mathbf
P,\mathfrak{b},\Lambda][L_2,\ldots, L_r]$ of products of terms of the form $F_{g',n'}\big[\begin{smallmatrix}a'_1 & \cdots & a'_{n'} \\ k'_1 &
  \cdots &
  k'_{n'} \end{smallmatrix}\big]$ with $2g'-2+n' \leq 2g-2+n$ and
$k'_1 +\cdots +k'_{n'} \leq k_1+\cdots+k_n$. Indeed, the condition
$\sum_{i\in{I}_m}i+\sum_{i'\in{I}'_m}i' \leq  k_1+\cdots+k_n$
follows from the vanishing condition for
$U_{k_1,h}^{a_1}\big[\begin{smallmatrix}{A} &A_1\cdots A_l \\ {I}&I_1
  \cdots I_l\end{smallmatrix}\big]$. Similarly,
\[ 2g_m-2+|I'_m|+|I_m|\leq 2\sum_m(g_m-1+|I_m|)+|I'_m|-|I_m| = 2g-h
  +|I'_m|-|I_m|\leq 2g-2+n, \]
and the equality holds if and only if $h=0$, $I'_m = \{k_2,\dots,k_n\}$ (thus, $I= \emptyset$), $|I_m| = 1$ and $g_m = g$. Since $d \geq 2$
these conditions imply that $l \geq 2$, therefore the vanishing
condition for $U_{k_1,h}^{a_1}\big[\begin{smallmatrix}{A} &A_1\cdots A_l \\ {I}&I_1
  \cdots I_l\end{smallmatrix}\big]$ implies that $I_m = \{y\}$ with
$y<k_1$. Thus, we conclude that for $F_{g',n'}\big[\begin{smallmatrix}a'_1 & \cdots & a'_{n'} \\ k'_1 &
  \cdots &
  k'_{n'} \end{smallmatrix}\big]$ with $2g'-2+n' = 2g-2+n$ we
necessarily have a strict inequality $k'_1 +\cdots +k'_{n'} <
k_1+\cdots+k_n$. Consequently, \eqref{eq:Frecursion} can be used to compute the $F_{g,n}\big[\begin{smallmatrix}a_1 & \cdots & a_n \\ k_1 &
  \cdots &
  k_n \end{smallmatrix}\big]$ recursively on $K=k_1+\cdots+k_n$ and
then on $\chi=2g-2+n$ in the lexicographic order and the claim $F_{g,n}\big[\begin{smallmatrix}a_1 & \cdots & a_n \\ k_1 &
  \cdots &
  k_n \end{smallmatrix}\big] \in \mathbb Q( \mathbf Q)[\mathbf
P,\mathfrak{b},\Lambda][L_2,\ldots, L_r]$ follows inductively.

To prove the specific form of the dependence on
$\Lambda$ let us simultaneously rescale $x^a_k\mapsto c^{-k} \,x^a_k$
and $\Lambda\mapsto c\,\Lambda$ for some $c\in\mathbb{C}^*$, which
corresponds to $\mathsf  J_k^a\mapsto c^k \mathsf  J_k^a$ and
$\widetilde{\mathsf W}^i_{k}\mapsto c^k\,\widetilde{\mathsf
  W}^i_{k}$. Then, as~\eqref{eq:WopAS}
 remains unchanged under this rescaling, so does the solution
 $\mathcal{Z}^L$. Thus, every $F_{g,n}\big[\begin{smallmatrix}a_1&a_2
   & \cdots & a_n \\ k_1&k_2 & \cdots & k_n \end{smallmatrix}\big]$ is
 a monomial in $\Lambda$ of degree $k_1+\cdots+ k_n$, which finishes
 the proof.
	\end{proof}

        As there are only a finite number of positive integers
        $k_1,\dots,k_n$ with the fixed sum $k_1+\cdots+k_n=k$, we
        have the following corollary of \cref{prop:Fgnbehaviour}.
	
	\begin{corollary}\label{cor:Fgnbehaviour}
		The expansion coefficients  $F_{g,n}$ of the partition function $\mathcal Z^L$ are elements of the ring $\mathbb Q(\mathbf Q) [\mathbf P, \mathfrak{b}, L_2,\ldots, L_r,\mathbf x]\llbracket \Lambda\rrbracket $.
	\end{corollary}

	\section{\texorpdfstring{$b$}{b}-Hurwitz numbers from the Airy structure}\label{sec:bHurwitzfromZ}
	
	In this section, we continue the study of the partition function $\mathcal Z^L$ of the shifted Airy structure. After performing various substitutions, we will prove that  $\mathcal Z^L$ reduces to the generating function $\tau^{(\bb)}_{G}$ of weighted $b$-Hurwitz numbers for an arbitrary rational weight $G(z)$. As a consequence, we will prove a stronger set of constraints for $\tau^{(\bb)}_{G} $, from which the cut-and-join equation derived in \cref{thm:rationalweight} follows.

\subsection{Reduction of the Airy structure} 
Recall that the partition function $\mathcal Z^L$ depends on the parameters $\mathbf P, \mathbf Q, \Lambda,\bb,\xx$ and the $\left(L_i\right)_{i\in[2..r]}$. We will prove  that $\mathcal Z^L$ reduces to $\tau^{(\bb)}_{G} $ for the weight $G(z) = \frac{\prod_{i=1}^r (P_i+z) }{\prod_{i=1}^{r-2}(Q_i-z)} $ after the following substitutions:
\begin{description}
	\item[S1] set $Q_0 = 0$,
	\item[S2] set all the variables $x^a_k = 0$ unless $a =1$, i.e.,
	\begin{equation*}\label{eq:xsub}
		x^a_k = 0,\qquad a \in [2..r] ,\quad k \in \mathbb Z_{\geq 1},
	\end{equation*} 
	\item[S3] for any $i\in[2..r]$, we set the $L_i$ to be the following:
	\begin{equation}\label{eq:Lambdasub}
		L_i  =  \sum_{1 \leq n_1 < \dots <n_{i} \leq r}
		\prod_{m=1}^{i}\big(-P_{r+1-n_m}-\hbar \mathfrak{b}(r-n_m+m-1)\big) - \mathsf V^i.
	\end{equation} 
\end{description}
When $\bb= 0$, it is known \cite{BychkovDuninBarkowskiKazarianShadrin2024} that these $G$-weighted Hurwitz numbers can be obtained by topological recursion on a certain spectral curve of degree $r$ (see \cref{sec:spectralcurve} for the explicit form of the curve). This perspective of topological recursion provides the following motivation for the $\bb= 0$ part of the substitutions {\bf S1}--{\bf S3}.
\begin{itemize}
\item[{\bf S1:}]  The mode $\mathsf J_0^a$ is related to the residue of $\omega_{0,1}$ at the $a$-th preimage of $x^{-1}(\infty)=(0,Q_1,..,Q_{r-2},\infty)$ from which we expect that $Q_0=0$. As the only condition on the Airy structure is that the $Q_0,\ldots, Q_{r-2}$ are pairwise disjoint, we may impose $Q_0 = 0$, as long as we assume that $Q_i \in \mathbb C^*$ for $i \in [r-2]$.
\item[{\bf S2:}] For the spectral curve given in \cref{sec:spectralcurve}, the expansion of the correlators $\omega_{g,n}(z_1,..,z_n)$ at $z_i=0$ produces the corresponding weighted Hurwitz numbers. On the other hand,  \cite{BorotChidambaramUmer2025} (see \cref{thm:tr1}) proves that the expansion of $\omega_{g,n}(z_1,..,z_n)$ at $z_i=0$ is related to the coefficients $F_{g,n}\big[\begin{smallmatrix}1 & \cdots & 1 \\ k_1 & \cdots & k_n \end{smallmatrix}\big]$ of $\mathcal Z^L$. In other words, the part of $\mathcal Z^L$ that survives after setting $x_k^a=0$ for all $k\geq1,a\in[2..r]$  gives  Hurwitz numbers.
\item[{\bf S3:}]  The choice of $L_i$ when $\bb = 0$ is again inspired by \cite{BychkovDuninBarkowskiKazarianShadrin2024} and \cite{BorotChidambaramUmer2025}, in particular the form of the (global) loop equations. 
\end{itemize}

The above observations do not help  determine the $\bb$-dependence of $L_i$, because refinements of the results of \cite{BychkovDuninBarkowskiKazarianShadrin2024,BorotChidambaramUmer2025} for arbitrary $\bb$ are still under investigation. The $\mathfrak{b}$-corrections are dictated by the form of the cut-and-join equation from \cref{thm:rationalweight} (see \cref{sec:Dlimits} for more details).
	
	It is worth emphasizing that the substitution {\bf S3
        }\eqref{eq:Lambdasub} for the $L_i$ is allowed  as the
        $F_{g,n}$ is proved to be a polynomial in $L_i$ in
        \cref{cor:Fgnbehaviour}. Thus, after substitution, we will not
        have arbitrarily negative powers in the $\mathbf P$ and
        $\mathbf Q$, and the resulting $F_{g,n}$ will belong to $\mathbb Q(\mathbf Q)[\mathbf P,\mathfrak b,\mathbf x]\llbracket \Lambda\rrbracket $.  We will denote the function obtained after the substitutions {\bf S1} and {\bf S3} simply by $\mathcal Z$, without a superscript. Then, $\mathcal Z $ has the form 
	\[
	\mathcal Z = \exp \left( \sum_{\substack{g \in \frac{1}{2} \mathbb Z_{\geq 0}, n \in \mathbb Z_{> 0}}} \hbar^{2g-2+n}  F_{g,n} \right)\,
	\] where $F_{g,n}$ is an element of $\mathbb Q(\mathbf Q)[\mathbf P,\mathfrak b,\mathbf x]\llbracket \Lambda\rrbracket $, and satisfies the constraints 
	\begin{equation}\label{eq:Wconst}
		\mathsf  W^i_k  \mathcal Z = \Omega_i \delta_{k,0}\mathcal  Z \qquad \text { for }  i \in [r] \,,\, k\geq 0\,,
	\end{equation}	where we define the $\Omega_i$ for $ i \in [r]$ as 
	\begin{equation}\label{eq:Lambdadef}
		\Omega_i  :=  \sum_{1 \leq n_1 < \dots <n_{i} \leq r}
		\prod_{m=1}^{i}\big(-P_{r+1-n_m}-\mathfrak{b}\hbar(r-n_m+m-1)\big)\,.
	\end{equation} Note that $\mathsf W^1_0 = \sum_{a=1}^r \J_0^a=\Omega_1$, and hence we have included the (trivial) constraint for $\mathsf W^1_0$ in \eqref{eq:Wconst}. It will also be convenient in the following to define $\Omega_0 :=1$. We will refer to the set of equations \eqref{eq:Wconst} as \textit{$\mathcal W$-constraints}. We will denote the function obtained after all the substitutions \textbf{S1}--\textbf{S3} by $\mathcal Z^{\operatorname{red}}$, and refer to it as the \textit{reduced partition function}. 
	
	\begin{remark}\label{rem:Whittaker}
		It is important to note that the state $\mathcal Z$ satisfying \eqref{eq:Wconst} is not the highest weight vector for the representation of ${\mathcal W}^{\mathsf k}(\mathfrak{gl}_r)$ that we are considering in this paper. Indeed, the highest weight vector in this representation is the partition function $1$ of the Airy structure $\mathcal I$ of \cref{prop:AS}. This vector $\mathcal Z$ is obtained by a shift of the zero modes $\mathsf W^i_0$ and we will refer to it as a Whittaker vector (extending the terminology used in \cite{SchiffmannVasserot2013, BorotBouchardChidambaramCreutzig2024}). Alternatively, the term singular vector is also used in the literature to describe such vectors.
	\end{remark}

\subsection{Constraints for $ \mathcal Z^{\operatorname{red}}$}

The goal of this section is to find a set of constraints for the reduced partition function $  \mathcal Z^{\operatorname{red}}$ which will determine it uniquely. Naively, one might want to apply the  substitutions \textbf{S1}--\textbf{S3} directly to the operators $\mathsf W^i_k$ appearing in the $\mathcal W$-constraints \eqref{eq:Wconst}. However, the resulting operators will not annihilate the reduced partition function $  \mathcal Z^{\operatorname{red}}$. Indeed, the recursive process of solving the $\mathcal W$-constraints \eqref{eq:Wconst} to obtain $\mathcal Z$ mixes together the various $a$, and hence, cannot directly be reduced to a recursive formula purely in terms of the $ F_{g,n}\big[\begin{smallmatrix} 1 & \cdots & 1 \\ k_1 & \cdots & k_n \end{smallmatrix}\big]$. 

Instead, we find a sequence of different operators $ \widehat{D}^r_k$, defined as certain combinations of the $\mathsf W^i_k$, that annihilate $\mathcal Z$. These operators $ \widehat{D}^r_k$  are constructed such that the substitution  \textbf{S2} commutes with the action of $ \widehat{D}^r_k$ on $\mathcal Z$. Then the operators $D_k$, defined as the reduction of $\widehat{D}^r_k$ under substitution \textbf{S2}, annihilate $\mathcal Z^{\operatorname{red}}$, and are closely related to the cut-and-join equation of \cref{thm:rationalweight}.

\subsubsection{The $\widehat{D} $-operators}

As before, we fix an integer $r \geq 2$. Let us begin by defining certain differential operators $\widehat{D}^i_k$. For $i = 0$, we set $\widehat D^{0}_{k} := 0 $ for any $k \in \mathbb Z$. Then we recursively define, for any $ i \in [r] $ and $ k \in \mathbb Z$, 
\begin{equation}\label{eq:Dhatdef}
	\widehat D^i_{k} := (-1)^{i+1} (\mathsf W^i_k - \delta_{k,0} \Omega_i) + \sum_{\substack{ k_2\geq 0 \\ k_1+k_2 = k}}\left(\J^1_{k_1} - \delta_{k_1,0} \hbar \mathfrak{b}(k_2+i-1)\right)\widehat D^{i-1}_{k_2}\,.
\end{equation} From the  definition of $\Omega_1$, it is easy to see that $ \widehat{D}^1_{0} = 0$.

\begin{lemma}\label{lem:DhatZ=0}
	We have the following constraints for the function $\mathcal Z$:
	\[
		\widehat{D}^i_k \mathcal Z = 0 \qquad \text { for }  i \in [r] \,,\, k\geq 0\,.
	\] 
\end{lemma}
\begin{proof}
The statement of the lemma follows directly from the the $\mathcal W$-constraints of equation \eqref{eq:Wconst}, and the definition \eqref{eq:Dhatdef} of the operators $\widehat D^i_{k}$. 
\end{proof}
In order to prove certain properties of the operators $\widehat{D}^i_k
$, we will find a combinatorial interpretation for them in terms of
weighted lattice paths. We will use the combinatorial formula for the
operators $\mathsf W^i_k$ in terms of weighted colored paths proved in
\cref{lem:WkCombinatorial}.  Recall the notion of bridges from
\cref{def:bridges}. Here is a combinatorial formula for the
$\widehat{D}$-operators, which is illustrated in \cref{fig:Paths2}.

\begin{figure}[h]
	\centering
	\includegraphics[width=0.9\textwidth
	]{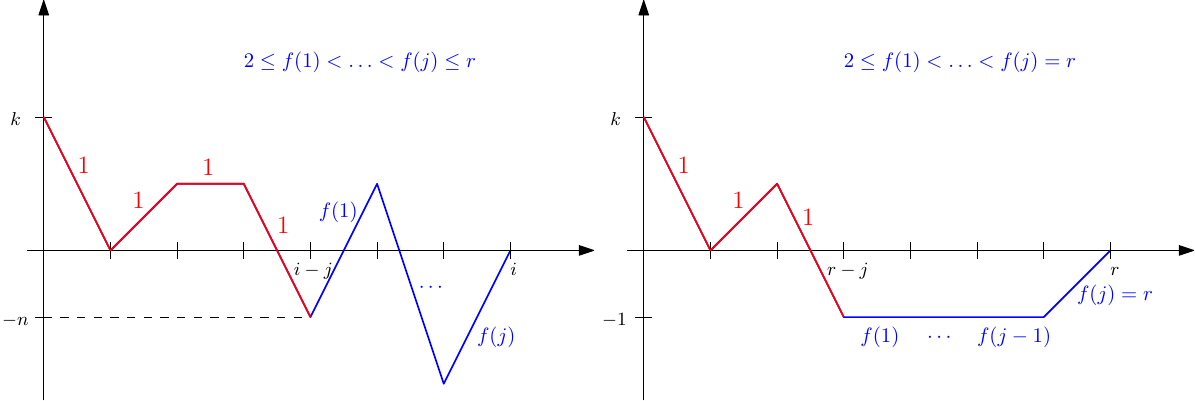}
	\caption{Left: A typical path contributing to the operator 
          $\widehat{D}^i_k$ from the last sum in
          \eqref{eq:DHatCombinatorial}. Right: A path contributing to the operator 
          $\widehat{D}^r_k$ from the last sum in
          \eqref{eq:DHatCombinatorial} that doesn't vanish after
          substituting $x^a_k=0$ for all $a \in [r], k \in \Z_{\geq
            1}$. Here the associated weight after the substitution is equal to $\J^{1}_{2}\J^{1}_{-1}\J^{1}_{2}(Q_{f(1)-1}-\bb\hbar(f(1)+j-3))\cdots (Q_{f(j)-1}-\bb\hbar(f(j-1)-1))(-1)^{r+1}\Lambda^{-1}$.}
	\label{fig:Paths2}
\end{figure} 

\begin{lemma}
	For any $i \in [r]$ and  $k \geq 0$, we  have the following combinatorial interpretation of the
	$\widehat{D}$-operators:
	\begin{align}
		\widehat{D}^{i}_k &= (-1)^{i+1}\sum_{\substack{\gamma \in \Gamma_{(0,k) \to
				(i,0)},\\ f\in \FI([i],[2..r])}}\wt(\gamma,f)
			+ \sum_{j=0}^{i} \sum_{\gamma \in \Gamma^{\geq0}_{(0,k) \to	(i-j,0)}} (-1)^{j}\Omega_{j} \prod_{u = 0}^{i-j}\wt_u(\gamma \cup \gamma_0,\mathbf 1) \nonumber\\
		&+\sum_{j=1}^{i-1}\sum_{n=1 }^\infty\sum_{\gamma \in \Gamma^{\geq0}_{(0,k) \to
				(i-j,-n)}}\sum_{\substack{\gamma' \in \Gamma_{(i-j,-n) \to
				(i,0)},\\ f\in \FI([j],[2..r])}}\wt(\gamma\cup\gamma',\tilde{f})\,, \label{eq:DHatCombinatorial}
	\end{align}
	where $\tilde{f}$ is the coloring of $\gamma\cup\gamma'$
	that colors each step of $\gamma$ by $1$, and the $\ell$-th
	step of $\gamma'$ by $f(\ell)$, and the coloring $\mathbf 1$
        colors each step of the associated path by $1$. The path
        $\gamma_0$ is a path that starts at $(i-j,0)$ and runs along
        the $X$-axis till $(i,0)$ for any $0\leq j\leq i$, and
        $\wt_0(\gamma,f) :=1$ by convention.
\end{lemma}
\begin{proof}
	We will prove the lemma by induction on $i$. When $ i = 1$, and $k\geq 1$ we see that \eqref{eq:DHatCombinatorial} reduces to the combinatorial formula \eqref{eq:WkCombinatorial} for the operators $\mathsf W^1_k$. As for the case $i=1$ and $k = 0$, recall that $\Omega_0 = 1$, and we directly obtain $\widehat{D}^{1}_0  = \sum_{a=1}^r \J^a_0 - \Omega_1 = 0 $. 
	
	Now assume that \eqref{eq:DHatCombinatorial}  holds for all $\widehat{D}^{\ell}_k $, such that $\ell \leq i-1$ and all $ k \geq 0$. Then, by plugging in the formula \eqref{eq:CombRecW} for the $\mathsf W^i_k$ into the definition  \eqref{eq:Dhatdef} of the $ \widehat{D}^{i}_k$, and using the inductive hypothesis for $ \widehat D^{i-1}_{k_2}$, we get the following formula:
	\begin{multline*}
	\widehat D^i_{k} = (-1)^{i+1} 
	\sum_{\substack{\gamma \in \Gamma_{(0,k) \to
				(i,0)},\\ f\in
                              \FI([i],[2..r])}}\wt(\gamma,f) + (-1)^i
                          \delta_{k,0} \Omega_i + (-1)^{i+1}
                          \sum_{\substack{k_1+k_2=k\\k_1<0}}
                          \left(\J^1_{k_1}-\delta_{k_1,0}\hbar\mathfrak{b}(k_2+i-1)\right)
                          \cdot \\
                          \cdot \sum_{\substack{\gamma \in \Gamma_{(0,k_2) \to
				(i-1,0)},\\ f\in \FI([i-1],[2..r])}}\wt(\gamma,f) +  \sum_{\substack{ k_2\geq 0 \\ k_1+k_2 = k}}\left(\J^1_{k_1} - \delta_{k_1,0} \hbar \mathfrak{b}(k_2+i-1)\right) \cdot\left( \sum_{j=1}^{i-1} (-1)^{j}\Omega_{j}\cdot \right. \\
	\cdot \left.\sum_{\gamma \in \Gamma^{\geq0}_{(0,k_2) \to	(i-1-j,0)}} \prod_{u = 0}^{i-1-j}  \wt_u(\gamma \cup \gamma_0,\mathbf 1 ) +\sum_{j=1}^{i-2}\sum_{n=1 }^\infty\sum_{\gamma \in \Gamma^{\geq0}_{(0,k_2) \to
			(i\sminus 1\sminus j,\sminus n)}}\sum_{\substack{\gamma' \in
                        \Gamma_{(i\sminus 1\sminus j,\sminus n) \to
				(i,0)},\\ f\in \FI([j],[2..r])}}
                          (\sminus 1)^{j}\wt(\gamma\!\cup\!\gamma',\tilde{f}) \right),
	\end{multline*} where $\tilde{f}$ is the coloring that colors all the steps of $\gamma$ by $1$ and the $l$-th step of $\gamma'$ by $f(l)$ and the path $\gamma_0 $ is the path that starts from $(i-1-j,0)$ and runs along the $X$-axis until $(i-1,0)$. The factor $\left(\J^1_{k_1}-\delta_{k_1,0}\hbar\alpha_0(k_2+i-1)\right) $ is the weight associated to the first step of a path that starts from $(0,k) $ and ends at $(i,0)$, such that its first step has increment $-k_1$, and hence we will concatenate this step with the paths that appear to the right (after shifting those paths by one unit in the positive $X$-direction) to obtain the following expression:
	\begin{multline*}
		\widehat D^i_{k} = (-1)^{i+1} 
		\sum_{\substack{\gamma \in \Gamma_{(0,k) \to
				(i,0)},\\ f\in
                              \FI([i],[2..r])}}\wt(\gamma,f)+\sum_{j=1}^{i} \sum_{\gamma \in \Gamma^{\geq0}_{(0,k) \to	(i-j,0)}} (-1)^{j}\Omega_{j} \prod_{u = 0}^{i-j} \wt_u(\gamma \cup \gamma_0,\mathbf 1) \\ 
		+ 	\sum_{n=1 }^\infty\sum_{\gamma \in
                  \Gamma^{\geq 0}_{(0,k) \to (1,\sminus n)}}\sum_{\substack{\gamma' \in \Gamma_{(1,\sminus n) \to
				(i,0)},\\ f\in
                              \FI([i\sminus 1],[2..r])}}(-1)^{i-1}\wt(\gamma\cup\gamma',\tilde{f})  + \\
		+  \sum_{j=1}^{i-2}\sum_{n=1 }^\infty\sum_{\gamma \in \Gamma^{\geq0}_{(0,k) \to
                    (i-j,-n)}}\sum_{\substack{\gamma' \in
                    \Gamma_{(i\sminus j,\sminus n) \to
				(i,0)},\\ f\in
                              \FI([j],[2..r])}} (-1)^{j} \wt(\gamma\cup\gamma',\tilde{f}),
	\end{multline*}  where $\gamma_0$ is now the path along the $X$-axis from $(i-j,0)$ to $(i,0)$. We have  combined the term $(-1)^i \delta_{k,0} \Omega_i   $ with the sum over the rest of the $\Omega_j$ to get the above expression. The sum over $j$ in the last line, can be expanded to include $j = i-1$, which absorbs the second to last line into the last line, and the resulting expression matches \eqref{eq:DHatCombinatorial}, which concludes the proof.
\end{proof}

\subsubsection{The operators $D_k$}

Recall that one of the motivations for defining the new  $\widehat{D}$-operators is precisely to be able to do the substitutions $(x^a_\ell)_{a \in[2..r],\ell\geq1} = 0$ and obtain operators that annihilate the reduced partition function $ \mathcal{Z}^{\operatorname{red}}$. For this purpose, we define new operators $D_k $ as follows. For any $k \geq 0$, define
\begin{equation}\label{eq:defDk}
	D_k := \Lambda \, \left( \widehat{D}^r_k\big|_{\,x^a_\ell = 0  \text{ for } a \in [2..r] \,,\, \ell \geq 1 }\right).
\end{equation} The operators $D_k$  depend on the  integer $ r\geq 2$, but we omit this for notational simplicity. 

We show that the operators $D_k$ annihilate the reduced partition function $\mathcal Z^{\operatorname{red}}$. During the course of the proof, we also prove an explicit combinatorial formula for $D_k$.

\begin{proposition}\label{prop:Dkprops} Consider the operators $D_k$ defined in \eqref{eq:defDk}.
	\begin{enumerate}[label=(\alph*)]
		\item The reduced partition function $\mathcal Z^{\operatorname{red}}$ satisfies the following constraints,
		\[
		D_k \mathcal Z^{\operatorname{red}} = 0 \qquad \text{
                  for } k \geq 0 \,.
		\] 
		\item The operators $D_k$ take the following form for any $k \geq 0$:
                  \begin{multline}
                    \label{eq:Dksimpler}
		D_k = \Lambda \sum_{j=0}^{r} \sum_{\gamma \in
                  \Gamma^{\geq0}_{(0,k) \to	(r\sminus j,0)}}
                \prod_{u = 0}^{r-j}
                  \wt_u(\gamma\!\cup\!\gamma_0,\mathbf 1)
                \sum_{1 \leq n_1<\cdots<n_{j}\!\leq r}
		\prod_{m=1}^{j}\big(P_{r\sminus n_m\splus 1}\splus
                \mathfrak{b}\hbar(r\sminus n_m\splus m\sminus 1)\big) 
		+ \\
		+(-1)^{r+1}\sum_{j=0}^{r-1} \sum_{\gamma \in
                  \Gamma^{\geq0}_{(0,k) \to	(r\sminus j,\sminus
                    1)}}\prod_{u = 0}^{r-j}
                  \wt_u(\gamma\!\cup\!\gamma_0,\mathbf 1) 
                \sum_{1 \leq n_1 < \cdots <n_{j\sminus 1} \leq r\sminus
                2}
		\prod_{m=1}^{j}\big(\sminus Q_{r\sminus n_m\sminus 1}+
                \mathfrak{b}\hbar(r\sminus n_m\splus m\sminus 2)\big),
			\end{multline}
			where $\gamma_0$ is the path that runs along the $ X $-axis from $(r-j,0) \to (r,0)$, and $\gamma_{-1}$ is the path that is horizontal from  $(r-j,-1)$ to $(r-1,-1)$ and then goes up to $(r,0)$.
	\end{enumerate}
\end{proposition}

\begin{proof}
	We will prove that for all $ k \geq 0 $, 
	\begin{equation}\label{eq:DZred}
		D_k \mathcal Z^{\operatorname{red}}  =\left( \widehat{D}^r_k \mathcal Z \right)\bigg|_{ x^a_\ell = 0 \text{ for } a \in [2..r] \,,\, \ell \geq 1},
	\end{equation}
	which implies part (a) of the proposition. Indeed,  \cref{lem:DhatZ=0} specialized to the case $i = r$  states that $\widehat{D}^r_k \mathcal Z = 0$ for all $k\geq 0$.
	
	Notice first that the reduction of the modes for $a \neq 1$ does not affect the action of the $\J^1_k$, i.e., we have
	\[
	\left(\J^1_{k_1}\cdots \J^1_{k_n} \mathcal Z\right)\bigg|_{ x^a_\ell = 0 \text{ for } a \in [2..r] \,,\, \ell \geq 1 }= \J^1_{k_1}\cdots \J^1_{k_n} \mathcal Z^{\operatorname{red}}\,.
	\] With this in mind, let us look at the three terms appearing in formula \eqref{eq:DHatCombinatorial} for $\widehat{D}^r_k$, which could contribute to the RHS of \eqref{eq:DZred}. The first term in \eqref{eq:DHatCombinatorial} vanishes identically when $i  = r$. Only paths where the coloring is always $1$ appear in the second term, and hence survives in its entirety. 
	
	Let us analyze the third term now, which involves paths
        $\gamma' \in \Gamma_{(j,-n) \to (r,0)}$  colored by
        $[2..r]$. As $n \geq 1$, all these paths $\gamma'$  start
        strictly below the $X$-axis, and as they end on the $X$-axis,
        they must touch the $X$-axis at some point. The segment that
        goes up and touches the $X$-axis will then contribute the
        weight $\J^a_{-\ell}$ for some $a \in[2..r]$ and $\ell \geq
        1$. Unless $a = r$ and $\ell = 1$, in which case $\J^r_{-1} =
        x^r_1 +\frac{(-1)^r}{\Lambda}$ (see \eqref{eq:Jrep}), all
        these terms vanish under the substitution $x^a_\ell = 0 $ for
        $a\in[2..r],\ell\geq1$. This means that the only paths
        $\gamma'$ that survive are such that $n = 1$, are horizontal
        from $(r-j,-1)$ to $(r-1,-1)$, and the last step $\gamma'_r$
        is colored by $r$ and connects $(r-1,-1)$ to $(r,0)$ (such a
        path is illustrated on the right side in~\cref{fig:Paths2}). The weights of these surviving paths are polynomials in the $\J^1_{\ell}$ and hence are not affected by the substitution, as noted before. This proves \eqref{eq:DZred}, and hence part (a) of the lemma. 
	
	As for part (b) of the proposition, the preceding analysis gives all the terms in  $\widehat{D}^r_k$ that survive after the reduction $x^a_\ell = 0$. Collecting them together, we get
		\begin{multline*}
	\frac{D_k}{ \Lambda} = 	\sum_{j=0}^{r} \sum_{\gamma \in
          \Gamma^{\geq0}_{(0,k) \to	(r\sminus j,0)}} (\sminus 1)^{j}\Omega_{j}  \prod_{u=0}^{r-j} \wt_u(\gamma\cup\gamma_0,\mathbf 1)+ \sum_{j=1}^{r-1}\sum_{\substack{\gamma \in \Gamma^{\geq0}_{(0,k) \to
				(r\sminus j,\sminus 1)},\\ f \in
                              \FI([j\sminus 1],[2..r\sminus 1]}}(-1)^{j}\wt(\gamma\cup\gamma_{-1},\tilde{f}_{-1})\bigg|_{x^r_1 = 0}\,,
	\end{multline*} where $\gamma_{-1}$ is the path that is horizontal from  $(r-j,-1)$ to $(r-1,-1)$ and then goes up to $(r,0)$, and $\tilde{f}_{-1}$ is the coloring that colors all the steps of $\gamma$ by $1$, the $\ell$-th step of $\gamma_{-1}$ by $f(\ell)$ for any $\ell \in [j-1]$, and the last step of $\gamma_{-1}$ by $r$. Now, we can extract the contributions coming from $\gamma_{-1}$ explicitly, which gives 
		\begin{multline*}
		D_k = 	\Lambda   \sum_{j=0}^{r} \sum_{\gamma \in
                  \Gamma^{\geq0}_{(0,k) \to	(r-j,0)}}
                (-1)^{j}\Omega_{j} \prod_{u = 0}^{r-j} \wt_u(\gamma
                \cup \gamma_0,\mathbf 1)+(-1)^{r+1} \sum_{j=1}^{r-1}
                \cdot\\
                \cdot\sum_{\gamma \in \Gamma^{\geq0}_{(0,k) \to
				(r-j,-1)}} \left( \prod_{u = 0}^{r-j} \wt_u(\gamma\cup\gamma_{-1},\mathbf 1) \right) \sum_{1 \leq n_1<\cdots<n_{j-1} \leq r-2 } \prod_{m=1}^{j-1} \left(-Q_{r-n_m-1}+\hbar \mathfrak{b} (r+m-n_m-2)\right) \,,
	\end{multline*} where we have set $n_m = r - f(j-m)$. Finally we plug in the expression \eqref{eq:Lambdadef} for the $\Omega_j$, and $\Omega_0 = 1$ to get the expression \eqref{eq:Dksimpler} stated in part (b) of the proposition.
\end{proof}

Having proved that the operators $D_k$ annihilate the reduced partition function $\mathcal Z^{\operatorname{red}}$, we would like to understand them better. We will find a simpler expression for the $D_k$  which will help us compare these operators with the cut-and-join operators. Before that, we need to state a technical lemma which will be useful in the proof of the simplification.

 \begin{lemma}\label{lem:xyidentity}
	We have the following identity between polynomials in $u_1,\dots,u_\ell,w_1,\dots,w_\ell$:
	\begin{equation}\label{eq:3DimStirling}
		\prod_{k=1}^\ell (u_k+w_k) = \sum_{I \sqcup J = [\ell]} \prod_{i \in I}(u_i+|J \cap [i]|) \prod_{j \in J} (w_j-(|J \cap [j-1]|))\,.
	\end{equation} 
\end{lemma}

\begin{proof}
	We will prove the lemma by induction on $\ell$. For $\ell=1$, the statement is obvious. Let $F_\ell(\mathbf u,\mathbf w)$ denote the polynomial that appears on the RHS of \eqref{eq:3DimStirling}. For any $\ell > 1$,  $F_\ell(\mathbf u,\mathbf w)$ can be split into two parts, depending on whether $1$ appears in $I$ or $J$:
	\begin{multline*}
		u_1 \sum_{I \sqcup J = [2..\ell]} \prod_{i \in
			I}(u_i+|J \cap [i]|)\prod_{j \in J}(w_j-(|J \cap [j-1]|))+\\
		+w_1  \sum_{I \sqcup J = [2..\ell]} \prod_{i \in
			I}(u_i+1+|J\cap [i]|)\prod_{j \in J}(w_j-(1+|J \cap [j-1]|)) = \\
		=u_1F_{\ell-1}(\mathbf u',\mathbf w')+ w_1F_{\ell-1}(\mathbf u'',\mathbf w''),
	\end{multline*}
	where we define $u_i' = u_{i+1}$, $u_i''=u_{i+1}+1$, $w_i' = w_{i+1}$, $w_i'' =	w_{i+1}-1$ for $i \in [\ell-1]$. Then, using the inductive hypothesis for $F_{\ell-1}$, we get 
	\[ 
	F_\ell(\mathbf u, \mathbf w) = u_1\prod_{k=2}^\ell(u_k+w_k)+w_1
	\prod_{k=2}^\ell(u_k+w_k) = \prod_{k=1}^\ell(u_k+w_k),
	\]
	which completes the proof.
\end{proof}
 
 Now, we prove a simplified formula for the operators $D_k$ in terms of weighted paths with the weight $\widetilde{\operatorname{wt}}$ (defined in \eqref{eq:tildewt}) instead of the weight $\operatorname{wt}$.

 \begin{lemma}
 	For any $k \geq 0$, we have the following combinatorial formula for the operators $D_k$:
 		\begin{equation}
 		\label{eq:combinatorial2}
 		D_k =  \Lambda \sum_{\gamma \in \Gamma^{\geq 0}_{(0,k)\to (r,0)}}\widetilde{\wt}\left(\gamma| \left(\mathbf P \right)\right) + (-1)^{r-1} \sum_{\gamma \in \Gamma^{\geq 0}_{(0,k)\to (r-1,-1)}}\widetilde{\wt}\left(\gamma|\left(-\mathbf Q,0\right) \right),
 	\end{equation} where we define $\left(\mathbf P\right) := \left(P_1,\ldots, P_{r}\right) $ and $\left(-\mathbf Q,0\right) := \left(-Q_1,\ldots, -Q_{r-2},0\right)$.
 \end{lemma}
 \begin{proof}
 	Recall equation \eqref{eq:Dksimpler}  for the operators $D_k$. Let us consider the first term that depends on $\mathbf P$. Consider a path $\gamma \in \Gamma^{\geq 0}_{(0,k) \to (r,0)}$, and let $ 1\leq n_1<\cdots< n_m \leq r$ denote the positions of the horizontal steps in $\gamma$, i.e., $\gamma$ has a  segment $s_n = (1,0)$ if and only if $n \in \{n_1,\ldots, n_m\}$. Let $s_{n_i}$ connect $(n_i-1, y_i)$ to $(n_i,y_i) $. For any subset $J \subset [m]$, we define the path $\gamma_J \in \Gamma^{\geq 0}_{(0,k) \to (r-|J|,0)}$ as the path obtained from $\gamma$ by removing all the horizontal steps $s_{n}$ for any $n =n_j$ with $ j \in J $. Then, we claim that the associated weights coincide:
 	\begin{equation}\label{eq:claim}
 		\widetilde{\wt}\left(\gamma|\left(\mathbf P\right)^{\operatorname{op}}\right) =  \sum_{J \subset [m]} \prod_{\ell=1}^{r-|J|} \wt_\ell(\gamma_J\cup \gamma_0, \mathbf 1) \prod_{j \in J}  \left(P_{r+1-n_j}+ \hbar \mathfrak{b}\left(r-n_j +\left|J\cap[j-1]\right|\right)\right),
 	\end{equation} where $\gamma_0 $ is the  path running along
        the $X$-axis from $(r-j,0)$ to $(r,0)$, and we use the
        notation $  \left(\mathbf P\right)^{\operatorname{op}}= \left( P_r, P_{r-1},\ldots, P_1\right) $. Notice that the contributions from the weights $\wt$ and $\widetilde{\wt}$ coincide for non-horizontal steps. Thus, in order to prove the claim, it suffices to prove the contributions to \eqref{eq:claim} for horizontal steps. Hence, \eqref{eq:claim} is equivalent to the following equation:
  \begin{equation*}
  	 \frac{\prod_{i=1}^m \left(P_{r+1-n_i} \sminus \hbar
             \mathfrak{b} y_i) \right)}{ (-\hbar \mathfrak{b})^m} =
         \sum_{I \sqcup J = [m]}   \prod_{i\in I}  \big( y_i +
         r-n_i+|J\cap[i]| \big)  \prod_{j \in J}
         \left(\frac{-P_{r+1-n_j}}{\hbar \mathfrak{b}} \sminus r \splus n_j \sminus\left|J\cap[j-1]\right|\right),
  \end{equation*} which is precisely the content of \cref{lem:xyidentity} with the choice $w_j = \frac{-P_{r+1-n_j}}{\hbar \mathfrak{b}}- r+ n_j$ and $u_i =  y_i + r-n_i$.
  
  Finally, by considering all possible choices of $m$,  we get a
  bijection between the set of paths $\gamma \in \Gamma^{\geq
    0}_{(0,k) \to (r,0)} $ that appears in \eqref{eq:combinatorial2},
  and the set of paths $\gamma_J \in \Gamma^{\geq 0}_{(0,k) \to
    (r-j,0)}$ along with the choice of an integer $ 0\leq j\leq r$ and
  integers $ 1 \leq n_1 < \cdots < n_j \leq r $ appearing in the first
  term of \eqref{eq:Dksimpler}. Thus, summing \eqref{eq:claim} over
  all paths $\gamma \in \Gamma^{\geq 0}_{(0,k) \to (r,0)} $ and
  shifting $n_j \to r+1-n_j$ on the LHS, gives that the sum
  $\sum_{\gamma \in \Gamma^{\geq 0}_{(0,k)\to (r,0)}}\widetilde{\wt} \left(\gamma|\left(\mathbf P\right)\right) $ is equal to 
 \[
\sum_{j=0}^{r} \sum_{\gamma \in \Gamma^{\geq0}_{(0,k) \to	(r-j,0)}} \prod_{u=1}^{r-j}  \wt_u(\gamma \cup \gamma_0,\mathbf 1) \sum_{1 \leq n_1 < \dots <n_{j} \leq r}
  	 \prod_{m=1}^{j}\big(P_{r+1-n_m}+ \hbar
         \mathfrak{b}(r - n_m+ m- 1)\big).\]
       A completely analogous argument works for the second term in \eqref{eq:combinatorial2} involving the $\mathbf Q$, which proves the lemma.
 \end{proof}

\subsection{Relation to $b$-Hurwitz numbers}

By comparing the explicit form of the  operators $D_k$ with the cut-and-join equation of \cref{thm:rationalweight}, we can prove one of our main results -- the function $\mathcal Z^{\operatorname{red}}$  matches the tau function $\tau^{(\bb)}_G$  of weighted $b$-Hurwitz numbers with the weight $G(z) = \frac{\prod_{i=1}^r (P_i+ z) }{\prod_{i=1}^{r-2}(Q_i-z)} $, after an identification between the variables $x^1_k$ and $\tilde p_k$. By taking appropriate limits, we can obtain the weighted $b$-Hurwitz generating function for arbitrary rational weights.

\subsubsection{$\mathcal Z^{\operatorname{red}}$ as a $b$-Hurwitz generating function}
We have the following theorem.
\begin{theorem}\label{thm:tauZidentification}
	We have the following equality between the reduced partition function $\mathcal Z^{\operatorname{red}}$ and the weighted b-Hurwitz tau function $\tau_G^{(\bb)}$, for weight $ G(z) = \frac{\prod_{i=1}^r (P_i+ z) }{\prod_{i=1}^{r-2}(Q_i-z)} $:
	\[
		 \tau^{(\bb)}_G = \mathcal Z^{\operatorname{red}}\big|_{x^1_k = \frac{\tilde p_k}{k},\, \Lambda = t} \,.
	\]
\end{theorem}
\begin{proof}
	Let us prove that the function $\mathcal  Z^{\operatorname{red}} \big|_{x^1_k = \frac{\tilde p_k}{k}} $ is a solution to the cut-and-join equation derived in \cref{thm:rationalweight}. Indeed, consider the operator
	\[
		C = \sum_{k\geq 0} \J^1_{-k-1} D_k.
	\] 
	Note that $\J^1_{-k-1}$ is the weight of a segment that starts at $(-1,-1)$ and ends at $(0,k)$. We can concatenate this segment with the paths that appear in \eqref{eq:combinatorial2} for the operators $D_k$, and shift this combined path one step to the right to obtain the following formula for the operator $C$:
		\begin{equation*}\label{eq:Coperator}
			C = \Lambda \sum_{\gamma \in \Gamma^{\geq 0}_{(0,-1)\to (r+1,0)}} \widetilde{\wt}\left(\gamma| \left(0,\mathbf P \right)\right)  
			+ (-1)^{r-1} \sum_{\gamma \in \Gamma^{\geq 0}_{(0,-1)\to (r,-1)}}\widetilde{\wt}\left(\gamma|\left(0,-\mathbf Q,0\right) \right),
		\end{equation*} where we denote $\left(0,\mathbf P\right):= \left(0, P_1,\ldots, P_r \right)$ and $\left(0,- \mathbf Q ,0\right):= \left(0,- Q_1 ,\ldots,- Q_{r-2} ,0 \right)$. Part (a) of \cref{prop:Dkprops} implies that $C \mathcal  Z^{\operatorname{red}} = 0 $. In this form, it is easy to recognize $C$ as  the cut-and-join operator of \cref{thm:rationalweight} after the substitutions $x^1_k = \frac{\tilde p_k}{k}$ for $k >0$ and $\Lambda = t$.

	 Moreover, \cref{cor:Fgnbehaviour} gives that $\mathcal
         Z^{\operatorname{red}}\big|_{x^1_k = \frac{\tilde p_k}{k},
           \Lambda = t}$ is an element of $ \Q(\QQ)[\PP, \tilde{\pp}, \bb](\!( \hbar)\!)\llbracket
t\rrbracket$ of the form $1 + O(t)$. Such a solution of the cut-and-join equation is proved to be unique in \cref{thm:rationalweight}, and hence we get the statement of the theorem.\end{proof}

 \subsubsection{Arbitrary rational weights}\label{sec:Dlimits} In this
 section, we explain how one can get constraints for $\tau^{(\bb)}_G$
 for an arbitrary rational weight, which we fix to be $G(z) =
 \frac{\prod_{i=1}^n\left(P_i+
     z\right)}{\prod_{i=1}^m\left(Q_i-z\right)}$ (where $m$ is not
 necessarily equal to $n-2$) by taking appropriate limits of the parameters $\mathbf Q, \mathbf P$.

       \begin{theorem}\label{thm:taufromZ}
 	The generating function  $\tau^{(\bb)}_G$ of
        $b$-Hurwitz numbers weighted by $G(z) =
        \frac{\prod_{i=1}^n\left(P_i+
            z\right)}{\prod_{i=1}^m\left(Q_i-z\right)}  $, for any
        integers $n,m \geq 0$, can be obtained from the reduced
        partition function $ \mathcal Z^{\operatorname{red}}$ for $r = \max(m+2,n)$ as
        \[ 
        \tau^{(\bb)}_G = \lim_{\substack{Q_{m+1},\ldots,Q_{r-2} \to \infty,\\ P_{n+1},\ldots,
            P_{r}\to \infty}}\mathcal
          Z^{\operatorname{red}}\big|_{\Lambda=t\frac{Q_{m+1}\cdots Q_{r-2}}{P_{n+1}\cdots
              P_{r}} \,,\, x^1_k = \frac{\tilde p_k}{k}}.
              \]
        In addition, $\tau^{(\bb)}_G$ satisfies the following constraints
 		\begin{equation}\label{eq:splitcutandjoin}
 			t \sum_{\gamma \in \Gamma^{\geq 0}_{(0,k)\to (n,0)}} \widetilde{\wt}\left(\gamma| \left(\mathbf P \right)\right)  \tau^{(\bb)}_{G}
 			+ (-1)^{m+1} \sum_{\gamma \in \Gamma^{\geq 0}_{(0,k)\to (m+1,-1)}}\widetilde{\wt}\left(\gamma|\left(-\mathbf Q,0\right) \right)  \tau^{(\bb)}_{G} = 0
 		\end{equation}
 		for any $k \geq 0$, where we denote $\left( \mathbf P \right) = \left( P_1,\dots,  P_n \right)$ and  $\left(- \mathbf Q,0 \right) = \left(- Q_1,\dots, - Q_{m},0 \right)$. 
      \end{theorem}

      \begin{proof}
This theorem follows from \eqref{eq:RationalRelation}. Indeed, let
$\check G(z)$ be the rational function associated with $G$ as in
\cref{rem:RationalRelation}, and let
$G'(z) := \frac{\prod_{i=1}^r\left(P_i+
            z\right)}{\prod_{i=1}^{r-2}\left(Q_i-z\right)}$, and
        $\check G'(z)$ be the associated function. Let $\PP_r :=
        (\PP,P_{n+1},\dots,P_r)$ and  $\QQ_{r-2} :=
        (\QQ,Q_{m+1},\dots,Q_{r-2})$. Then
        \begin{multline*}
          \tau^{(\bb)}_G(t,\PP,\QQ) = \tau^{(\bb)}_{\check G}( t\cdot \frac{P_1\cdots
          P_n}{Q_1\cdots
          Q_m}, \PP^{-1}, \QQ^{-1}) = \\
        =\tau^{(\bb)}_{\check G'}\left( t\cdot \frac{P_1\cdots
          P_n}{Q_1\cdots
          Q_m},
       \PP_r^{-1}, \QQ_{r-2}^{-1}\right)\bigg|_{\substack{Q^{-1}_{m+1},\ldots,Q^{-1}_{r-2}=0,\\P^{-1}_{n+1},\ldots,P^{-1}_{r}=0}}
        = \lim_{\substack{Q_{m+1},\ldots,Q_{r-2} \to \infty,\\ P_{n+1},\ldots,
            P_{r}\to \infty}}
          \tau^{(\bb)}_{G'}\left(t\frac{Q_{m+1}\cdots Q_{r-2}}{P_{n+1}\cdots
              P_{r}},\PP_r,\QQ_{r-2}\right)
        \end{multline*}
        which is equal to $\lim_{\substack{Q_{m+1},\ldots,Q_{r-2} \to \infty,\\ P_{n+1},\ldots,
            P_{r}\to \infty}}\mathcal
          Z^{\operatorname{red}}\big|_{\Lambda=t\frac{Q_{m+1}\cdots Q_{r-2}}{P_{n+1}\cdots
              P_{r}} \,,\, x^1_k = \frac{\tilde p_k}{k}}$ by
          \cref{thm:tauZidentification}.
          Similarly, $\tau^{(\bb)}_G$ satisfies \eqref{eq:splitcutandjoin} if
          and only if $\tau^{(\bb)}_{\check G}$ satisfies
          \eqref{eq:splitcutandjoin} after the change of
          variables as in \cref{rem:RationalRelation}. It was shown in the proof
          of~\cref{thm:rationalweight} that the latter is equal
          to
        \begin{multline*}\bigg[t\sum_{i=0}^n e_{i}(\PP)\sum_{\gamma \in \Gamma^{\geq 0}_{(0,k) \to
                  (i,0)}}\widetilde{\wt}(\gamma\,|0)- \sum_{i =0}^m e_{i}(-\QQ) \sum_{\gamma \in \Gamma^{\geq 0}_{(0,k) \to
                  (i+1,-1)}}\widetilde{\wt}(\gamma\,|0)\bigg]\tau^{(\bb)}_{\check G} = \\
              = \bigg[t\sum_{i=0}^r e_{i}(\PP_r)\sum_{\gamma \in \Gamma^{\geq 0}_{(0,k) \to
                  (i,0)}}\widetilde{\wt}(\gamma\,|0)- \sum_{i =0}^{r-2} e_{i}(-\QQ_{r-2}) \sum_{\gamma \in \Gamma^{\geq 0}_{(0,k) \to
                  (i+1,-1)}}\widetilde{\wt}(\gamma\,|0)\tau^{(\bb)}_{\check G'}\bigg]\bigg|_{\substack{Q_{m+1},\ldots,Q_{r-2} =0,\\ P_{n+1},\ldots,
            P_{r}=0}}
      \end{multline*}
      and we conclude as before by applying
      \cref{thm:tauZidentification}, and the expression \eqref{eq:combinatorial2} for the $D_k$.
        \end{proof}
 
This is a vast generalization of recent results that were limited to the cases $(n,m) = (0,1), (1,0),
 (2,0)$~\cite{BonzomChapuyDolega2023} and $(n,m) = (3,0)$~\cite{BonzomNador2023}. Notice that \cref{thm:taufromZ} proves a stronger set of constraint
 on the generating function $\tau^{(\bb)}_G$ as compared to
 \cref{thm:rationalweight}. Indeed multiplying  equation \eqref{eq:splitcutandjoin} on the left by $\J^1_{-k-1}$ and summing over $k \in \mathbb Z_{\geq 0} $ gives the cut-and-join equation proved in \cref{thm:rationalweight}.
 
 \begin{remark}\label{rem:AGT}
 	The  case of $G(z) = \frac{1}{\prod_{i=1}^{r-1}(Q_i - z)}$ can be treated without taking limits as in \cref{thm:taufromZ}. Instead, we could consider the  representation of $ \mathcal W^{\mathsf k}(\mathfrak{gl}_r)$ given by 
 	\begin{equation*}
 		\label{eq:Jrep'}
 		\begin{split}
 			\mathsf J^a_{k} = \left\{
 			\begin{array}{lr}
 				\hbar  \partial_{x^a_k}& k> 0,\\
 				\hbar (-k)x^{a}_{-k}  & k < 0, \\
 				Q_{a-1}-\hbar \mathfrak{b}(a-\delta_{a,1})  & a \in [r] ,k = 0\, 
 			\end{array}
 			\right.
 		\end{split}
 	\end{equation*} with $Q_0 = 0$ as before. Then  the Whittaker vector $\widehat{\mathcal Z}$ (whose existence and uniqueness follows from a slight modification of the Airy structure of \cite{BorotBouchardChidambaramCreutzig2024}) satisfying 
 	\[
 		\mathsf W^{i}_k \widehat{\mathcal Z }= \delta_{i,r}\delta_{k,1} t \widehat{\mathcal Z } \quad i \in [r], k \geq 1 \,,
 	\] is such that its reduction $\widehat{\mathcal Z }^{\operatorname{red}}$ (under the substitutions {\bf S1,S2}) matches the  generating function $\tau^{(\bb)}_G$ on the nose. We do not provide a proof, but the interested reader can  follow the combinatorial methods of the previous sections to derive this result. We mention this result as this Whittaker vector is a slight variant of the Gaiotto state that appears in the AGT correspondence \cite{AldayGaiottoTachikawa2010,SchiffmannVasserot2013, MaulikOkounkov2019, BorotBouchardChidambaramCreutzig2024}.
 \end{remark}
 
 \section{Topological recursion for weighted Hurwitz numbers}\label{sec:TR}
 In this section we prove that, when $\bb = 0$, rationally weighted Hurwitz numbers can be computed using   topological recursion \cite{EynardOrantin2007} on an associated spectral curve. This provides a totally different proof of one of the main theorems of \cite{BychkovDuninBarkowskiKazarianShadrin2024} using $\mathcal W$-algebra representations. Our proof relies on the  weighted $b$-Hurwitz interpretation of the Airy structure partition function $\mathcal Z$ proved in \cref{sec:bHurwitzfromZ}. Throughout this section, we set $\bb = 0$.

 \subsection{A spectral curve for $\mathcal Z$}\label{sec:spectralcurve}
 
 Let us give a very brief introduction to the topological recursion (TR) formalism.  TR takes as input data a \textit{spectral curve} $( \Sigma, x, \omega_{0,1}, \omega_{0,2}) $, where $\Sigma $ is a Riemann surface, $x$ is a  meromorphic function on $\Sigma$, $\omega_{0,1}$ is a meromorphic differential on $\Sigma$ and $\omega_{0,2}$ is a fundamental bidifferential on $\Sigma \times \Sigma$ (i.e., a symmetric meromorphic bidifferential, whose only poles consist of a double pole on the diagonal with biresidue 1). From an admissible\footnote{See \cite[Definition 2.5]{BorotBouchardChidambaramKramerShadrin2025} for a definition of admissibility.} spectral curve, TR\footnote{Throughout this paper, we understand topological recursion to mean the generalized topological recursion formula of \cite{BouchardHutchinsonLoliencarMeiersRupert2014, BouchardEynard2013} whenever $x$ has non-simple ramification points.} produces symmetric meromorphic differentials $ \left(\omega_{g,n} \right)_{g \in \mathbb  Z_{\geq 0}\,,\, n\in \mathbb Z_{\geq 1}} $  called \textit{correlators} on $\Sigma^{\times n}$, in the range $ 2g - 2 + n >0 $.  For  the explicit topological recursion formula,  various nice properties satisfied by the correlators and their appearance in many different contexts in geometry, see \cite{EynardOrantin2007,BouchardEynard2013}.

 We are interested in the  curve $\Sigma \subset \mathbb P^1_x \times \mathbb P^1_y $, cut out by the equation
 \begin{equation}\label{eq:curve}
 	\sum_{i=0}^r (-y)^{r-i}\left( \frac{e_i(P_1,\dots, P_r)}{(-x)^i}  +(-1)^r \frac{e_{i-1}(Q_0,Q_1,\dots, Q_{r-2})}{\Lambda \,x^{i-1}}\right) = 0,
 \end{equation} where we recall that $Q_0 $ has been set to $0$, but we keep using $Q_0$ for notational convenience. The curve $\Sigma$ admits a normalization $\tilde{\Sigma}$ of genus zero which has the following explicit parametrization (with uniformizing coordinate $z \in \mathbb P^1$)
 \begin{equation}\label{eq:algcurve}
 	x(z) = -\Lambda \frac{\prod_{i=1}^r (P_i +z)}{\prod_{i=0}^{r-2} (Q_i-z)},\qquad y(z) = - \frac{z \prod_{i=0}^{r-2} (Q_i-z)}{ \Lambda \prod_{i=1}^r (P_i + z)}.
 \end{equation} We also define $\omega_{0,1} = y dx$ and $\omega_{0,2}(z_1,z_2) = \frac{dz_1 dz_2}{(z_1-z_2)^2}$ to obtain the spectral curve $(\tilde{\Sigma}, x, \omega_{0,1}, \omega_{0,2} )$.  Notice that the fiber $  x^{-1}(\infty)$ is of rank $r$ and consists of the points $ z = Q_0,\ldots, Q_{r-2},\infty$. Then, we have the following statement of \cite[Theorem 2.3]{BorotChidambaramUmer2025}.
 \begin{theorem}\label{thm:tr1}
 	Consider the spectral curve $(\tilde{\Sigma}, x, \omega_{0,1}, \omega_{0,2} ) $ defined from \eqref{eq:algcurve}, and the corresponding correlators $\left(\omega_{g,n} (z_1,\ldots,z_n) \right)_{g \in \mathbb  Z_{\geq 0}\,,\, n\in \mathbb Z_{\geq 1}} $    constructed by topological recursion. The expansion coefficients of the correlators $\omega_{g,n}$ at $ x = \infty $ coincide with the $F_{g,n}$ of the partition function $\mathcal Z$ satisfying the $\mathcal W$-constraints \eqref{eq:Wconst}. More precisely, for $i \in [n]$, fix $1 \leq a_i \leq r $. Assuming that $z_i$ is near $ Q_{a_i-1}$ if $a_i \in [r-1]$ or that $z_i $ is near $\infty$ if $a_i = r$, we have  
 	\[
 		\frac{\omega_{g,n}(z_1,\dots, z_n)}{\dxz{1}\cdots \dxz{n}} - \frac{\delta_{g,0}\delta_{n,2}}{(x(z_1)-x(z_2))^2}  = \sum_{k_1,\ldots, k_n \geq 1} F_{g,n}\big[\begin{smallmatrix} a_1 & \cdots & a_n \\ k_1 & \cdots & k_n \end{smallmatrix}\big] \prod_{i=1}^n x(z_i)^{-k_i-1}\,,
 	\] under the following  assumptions that $\mathbf P, \mathbf Q$ are generic:
 	\begin{itemize}
 		\item $Q_0,\ldots, Q_{r-1}$ are pairwise distinct;
 		\item $x(z): \tilde{\Sigma} \to \mathbb P^1$ only has simple ramification points.
 	\end{itemize} 
 \end{theorem}
 
 We will view $(\tilde{\Sigma}, x, \omega_{0,1}, \omega_{0,2} )$ as a family of spectral curves over a base parametrized by $\mathbf P, \mathbf Q$ (see \cite[Section 5.1]{BorotBouchardChidambaramKramerShadrin2025} for a discussion on families of spectral curves). 
 Then, we will repeatedly use the result that the correlators are analytic in globally admissible families as proved in \cite[Theorem 5.8]{BorotBouchardChidambaramKramerShadrin2025}. Global admissibility of a family in our situation boils down to the verification of the three conditions \textit{(gA1)-(gA3)} given in Definition 4.9 of \textit{loc.cit.}

Putting together the above theorem and the $b$-Hurwitz interpretation of $\mathcal Z$ from \cref{sec:bHurwitzfromZ} proves that the $\omega_{g,n}$ are generating functions for certain weighted Hurwitz numbers.  We also use the results of \cite{BorotBouchardChidambaramKramerShadrin2025} to lift the assumptions on $\mathbf P, \mathbf Q$ being generic in \cref{thm:tr1}.
 
 \begin{proposition}\label{prop:TR1}
	Consider the spectral curve $(\tilde{\Sigma}, x, \omega_{0,1}, \omega_{0,2} ) $ defined from \eqref{eq:algcurve}, and corresponding correlators $\left(\omega_{g,n} (z_1,\ldots,z_n) \right)_{g \in \mathbb  Z_{\geq 0}\,,\, n\in \mathbb Z_{\geq 1}} $    constructed using the topological recursion formula. Assume that for all $i \in [n]$, $z_i $ is near $Q_0 = 0$. Under the identification $\Lambda = t$ we have 
	\[
	\frac{\omega_{g,n}(z_1,\dots, z_n)}{\dxz{1}\cdots \dxz{n}} - \frac{\delta_{g,0}\delta_{n,2}}{(x(z_1)-x(z_2))^2}  = 	\sum_{\mu_1,\dots,\mu_n \in \Z_{\geq
			1}}\HurGg(\mu_1,\dots,\mu_n)|\Aut(\mu)|\prod_{i=1}^n\mu_ix(z_i)^{-\mu_i-1}\,,
	\] for the weight $G(z) = \frac{\prod_{i=1}^r (P_i+ z) }{\prod_{i=1}^{r-2}(Q_i-z)}  $, where $\mathbf P, \mathbf Q \in \mathbb C^{*}$.
 \end{proposition}
 \begin{proof}
 	\cref{thm:tr1} shows that  expanding the correlators $\omega_{g,n}$ such that  $z_1,\dots, z_n$ are near $Q_0 = 0$, gives the coefficients $F_{g,n}\big[\begin{smallmatrix} 1 & \cdots & 1 \\ k_1 & \cdots & k_n \end{smallmatrix}\big]$. These coefficients are precisely the ones that appear in the reduced partition function $\mathcal Z^{\operatorname{red}}$. Then, the specialization of \cref{thm:tauZidentification} to $\bb = 0$  proves that $\tau^{(0)}_G $ coincides with $\mathcal Z^{\operatorname{red}}$ for the weight $ G(z) = \frac{\prod_{i=1}^r (P_i+ z) }{\prod_{i=1}^{r-2}(Q_i-z)}  $, up to the identification $\Lambda = t$. More precisely, extracting the expansion coefficients of $\tau^{(0)}_G $ and $\mathcal Z^{\operatorname{red}}$ yields the following equality
 	\[
 		F_{g,n}\big[\begin{smallmatrix} 1 & \cdots & 1 \\ \mu_1 & \cdots & \mu_n \end{smallmatrix}\big] = \HurGg(\mu_1,\dots,\mu_n)|\Aut(\mu)|\prod_{i=1}^n\mu_i\,
 	\] which proves the statement of the proposition when $\mathbf P, \mathbf Q$ are generic.
 	
 	In order to remove the restriction that $Q_1,\ldots, Q_{r-2} $ are pairwise disjoint, we view our spectral curve defined using \eqref{eq:curve} as a family over the base $\left(\mathbb C^*\right)^{r-2}$ parametrized by $ Q_1,\dots, Q_{r-2}$. Then, we will show that this family is globally admissible to conclude that the  correlators $\omega_{g,n}$ are analytic  \cite[Theorem 5.8]{BorotBouchardChidambaramKramerShadrin2025}. 	If  $Q_1,\ldots, Q_{r-2} $ are pairwise disjoint, the curve $\Sigma$ from \eqref{eq:curve} is easily seen to be smooth, and hence automatically globally admissible. On the other hand, if the $Q_i$ are not pairwise disjoint, the curve acquires a single singular point at $x = \infty , y = 0$ and hence satisfies \textit{(gA1)}. At any preimage in $ \tilde{\Sigma} $ of this singular point, the local parameter is  $\bar{s} = 0$, and thus satisfies the local admissibility condition  \textit{(gA2)}. Finally for \textit{(gA3)}, condition \textit{C-ii} is always satisfied when $\bar{s} =0 $, and  checking the non-resonance condition for points with $\bar{s} = 0$ is a straightforward computation.
 	
 	The last restriction that $x(z)$ has simple ramification points, can be lifted by using the  Bouchard--Eynard topological recursion formula instead of the Eynard--Orantin topological recursion formula - see \cite[Section 6.3.1]{BorotBouchardChidambaramKramerShadrin2025} for a proof.
 \end{proof}
 
 \subsection{TR for arbitrary rational weights} Now we extend the results of the previous subsection to prove topological recursion for weighted Hurwitz numbers with arbitrary rational weights. We will take appropriate limits of the statement proved in \cref{prop:TR1}, using the results of \cite{BorotBouchardChidambaramKramerShadrin2025}. Consider the weight 
 \[
 	G(z) = \frac{\prod_{i=1}^p (P_i+z)}{\prod_{i=1}^q (Q_i-z)}
 \] where $p,q \geq 0$. Consider the curve $ S $  defined by 
 \begin{equation} \label{eq:curve2}
	x(z) =  t \frac{\prod_{i=1}^p (P_i +z)}{z \prod_{i=1}^{q} (Q_i-z)},\qquad y(z) = \frac{z^2 \prod_{i=1}^{q} (Q_i-z)}{ t \prod_{i=1}^p (P_i + z)}.
\end{equation}  As before, define $\omega_{0,1} = y dx$ and $\omega_{0,2}(z_1,z_2) = \frac{dz_1 dz_2}{(z_1-z_2)^2}$ to obtain the spectral curve $(S, x, \omega_{0,1}, \omega_{0,2} )$. 

\begin{theorem}\label{thm:TRmain}
	Consider the correlators $ \left(\omega_{g,n} (z_1,\ldots,z_n) \right)_{g \in \mathbb  Z_{\geq 0}\,,\, n\in \mathbb Z_{\geq 1}}  $ constructed using topological recursion on the spectral curve  $(S, x, \omega_{0,1}, \omega_{0,2} )$. When $z_1,\ldots,z_n$ are near $0$, we have 
	\[
		\frac{\omega_{g,n}(z_1,\dots, z_n)}{\dxz{1}\cdots \dxz{n}} - \frac{\delta_{g,0}\delta_{n,2}}{(x(z_1)-x(z_2))^2}  = 	\sum_{\mu_1,\dots,\mu_n \in \Z_{\geq
				1}}\HurGg(\mu_1,\dots,\mu_n)|\Aut(\mu)|\prod_{i=1}^n\mu_ix(z_i)^{-\mu_i-1}\,,
		\]  for the weight $G(z) = \frac{\prod_{i=1}^p (P_i+z)}{\prod_{i=1}^q (Q_i-z)}  $.
\end{theorem}
\begin{proof} The case of $q = p-2  $ is precisely the content of \cref{prop:TR1} for $r = p$. We  need to address the remaining cases $ 	q> p-2$ and $ q <p-2$ separately. 
	
	\textit{The case  $ 	q < p-2 $: } Consider the curve
        \eqref{eq:curve} with the parameter $r = p $,  set $\Lambda  =
        t Q_{q+1} \cdots Q_{p-2}$, and denote the resulting curve by
        $\Sigma_1$. The normalization of the curve obtained in  the
        limit $Q^{-1}_{q+1}, \ldots, Q^{-1}_{p-2} \to 0 $ of
        $\Sigma_1$ is precisely the curve $ S$.  Thus, if the
        correlators $\omega_{g,n}$ were analytic,  applying the above
        limit  to  \cref{prop:TR1} for the curve
        $(\tilde{\Sigma_1},x,\omega_{0,1},\omega_{0,2})$ would yield
        the statement. The analyticity for Hurwitz numbers is
        immediate from the definition, as explained in the proof of
        \cref{thm:taufromZ}. 
	
	 Let us prove the analyticity of the correlators $\omega_{g,n}$ using  \cite[Theorem 5.8]{BorotBouchardChidambaramKramerShadrin2025}. To apply this theorem, we need to check that the  spectral curve $(\tilde{\Sigma}_1,x,\omega_{0,1},\omega_{0,2})$ viewed as  a family over the base $\mathbb C^{p-q-2}$ parametrized by $Q^{-1}_{q+1}, \ldots, Q^{-1}_{p-2}$ is  globally admissible. If all the $Q^{-1}_{q+1}, \ldots, Q^{-1}_{p-2} \neq 0$, the curve is globally admissible as proved in the last paragraph of the proof of \cref{prop:TR1}. If some of the $ Q^{-1}_{q+1}, \ldots, Q^{-1}_{p-2}$ go to $0$,  the only possible singularity of $\Sigma_1$ is still $(x,y) = (\infty,0)$ and thus condition \textit{(gA1)} is satisfied. However, the structure of the singularity changes -- the point  $z = 0$ in $\tilde{\Sigma}_1$ is now a preimage of this singular point.  At $z=0,$ the local parameter $\bar{s} = -1$, and the local admissibility condition \textit{(gA2)} is still satisfied. Finally for condition $ \textit{(gA3)}$, condition $C-ii$ is satisfied automatically when $\bar{s} \leq  0$ (and the non-resonance for points with $\bar{s} = 0$ is again an easy calculation).
	
	\textit{The case  $ 	q > p-2 $: } Consider the curve \eqref{eq:curve} with the parameter $r = q+2 $,  set $\Lambda  = t/( P_{p+1} \cdots P_{q+2})$, and denote the resulting curve by $\Sigma_2$. First, let us take the limit $  P^{-1}_{q+2} \to 0$. In this limit, the curve becomes reducible as is evident using the homogeneous coordinates $[X_0:X_1] = x$ and $[Y_0:Y_1]= y$,
		\begin{equation*}
		Y_1 \left(  \sum_{i=1}^{q+2} Y_1^{i-1} (-Y_0)^{q+2-i} X_1^{i-1} X_0^{q+2-i} \left(X_1\frac{e_{i-1}(P_1,\ldots, P_{q+1})}{P_{p+1}\cdots P_{q+1}} +(-1)^q X_0 \frac{e_{i-1}(Q_0,Q_1,\dots, Q_{q})}{t}  \right) \right) = 0\,.
	\end{equation*}  However, the extra component $Y_1 = 0$ (or $y =\infty$) does not contribute to topological recursion and can be discarded as explained in \cite[Section 6.2.2]{BorotBouchardChidambaramKramerShadrin2025}.  Finally, if we send the remaining parameters $ P^{-1}_{p+1} ,\ldots, P^{-1}_{q+1} \to 0 $ we recover the spectral curve $S$ of interest, and we can apply the analyticity result of \cite[Theorem 5.8]{BorotBouchardChidambaramKramerShadrin2025} to finish the proof. As the verification of  global admissibility is analogous to the previous case (the only difference being that $(x,y) = (0,\infty)$ is the singular point), we omit this.
\end{proof} 
The above theorem first appeared in \cite{BychkovDuninBarkowskiKazarianShadrin2024} and generalizes the results of \cite{AlexandrovChapuyEynardHarnad2020}. Here, we give a completely new proof using $\mathcal W$-algebra representations.

 \begin{remark}\label{rem:scconventions}
 	Typically, in papers relating topological recursion and weighted Hurwitz theory such as \cite{BychkovDuninBarkowskiKazarianShadrin2024,AlexandrovChapuyEynardHarnad2020},  \cref{thm:TRmain} is formulated for the following slightly different spectral curve:
 	\[
 	\tilde{x}(z) =   \frac{z \prod_{i=1}^{q} (Q_i-z)}{ t \prod_{i=1}^p (P_i +z)}, \qquad \tilde{y}(z) = \frac{ t \prod_{i=1}^p (P_i + z)} { \prod_{i=1}^{q} (Q_i-z)}.
 	\] Then,  Hurwitz numbers are obtained by expanding the corresponding correlators $\tilde{\omega}_{g,n}$ at $z= 0$ (thus, $\tilde{x} = 0$) in powers of $\tilde{x}$ (as we have presented in the introduction \cref{theo:TRRational}). However, notice that $\tilde{x} = 1/x$ and $\tilde{y} d \tilde{x} = -y dx$ where $ x,y$ define the curve $S$ \eqref{eq:curve2}. The  correlators $\tilde{\omega}_{g,n}$ and $\omega_{g,n}$ only differ by a sign $(-1)^n$ which is cancelled by the signs coming from $\frac{d \tilde{x} }{\tilde{x} }= - \frac{d x}{x} $ in the expansion.
 \end{remark}

\bibliographystyle{amsalpha}

\bibliography{biblio2015}

\end{document}